\documentclass[a4paper, 11pt]{amsart}
\usepackage{geometry}
\usepackage[latin1]{inputenc}
\usepackage[T1]{fontenc}
\usepackage{amssymb}
\usepackage{amsmath}
\usepackage{amsthm}
\usepackage{amssymb}
\usepackage{tikz-cd}
\usepackage[hidelinks]{hyperref}
\usepackage[nameinlink, capitalise]{cleveref}

\newtheorem{LetterTheorem}{Theorem}

\newtheorem{Theorem}{Theorem}[section]
\newtheorem{Proposition}[Theorem]{Proposition}
\newtheorem{Lemma}[Theorem]{Lemma}
\newtheorem{Corollary}[Theorem]{Corollary}
\theoremstyle{definition}
\newtheorem{Remark}[Theorem]{Remark}
\newtheorem{Definition}[Theorem]{Definition}
\newtheorem{Construction}[Theorem]{Construction}
\newtheorem{Example}[Theorem]{Example}

\newcommand{\rmod}{\operatorname{mod}}
\newcommand{\stmod}{\operatorname{\underline{mod}}}
\newcommand{\stgproj}{\operatorname{\underline{Gproj}}}
\newcommand{\proj}{\operatorname{proj}}
\newcommand{\Gproj}{\operatorname{Gproj}}

\newcommand{\CC}[1]{\mathcal{C}(\rmod #1 )}
\newcommand{\CCp}[1]{\mathcal{C}(\proj #1 )}
\newcommand{\CCpleft}[1]{\mathcal{C}(#1\text{-proj})}

\newcommand{\KKp}[1]{\mathcal{K}(\proj #1 )}
\newcommand{\KKptac}[1]{\mathcal{K}_{\mathrm{tac}}(\proj #1)}
\newcommand{\Pperp}{\!{}^\perp \mathcal{P}\!}

\newcommand{\HH}{\mathrm{H}}
\newcommand{\LL}{\mathcal{L}}
\newcommand{\FF}{\mathcal{F}}
\newcommand{\GG}{\mathcal{G}}
\newcommand{\ZZ}{\mathbb{Z}}
\newcommand{\PP}{\mathcal{P}\!}

\newcommand{\tleq}{\tau_{\leqslant 0}\,}

\newcommand{\Du}{\operatorname{D}}
\newcommand{\Tr}{\operatorname{Tr}}
\newcommand{\im}{\operatorname{Im}}
\newcommand{\Ker}{\operatorname{Ker}}
\newcommand{\Cok}{\operatorname{Cok}}

\newcommand{\ddim}{\operatorname{domdim}}
\newcommand{\Hom}{\operatorname{Hom}}
\newcommand{\Ext}{\operatorname{Ext}}
\newcommand{\Tor}{\operatorname{Tor}}
\newcommand{\rad}{\operatorname{rad}}
\newcommand{\soc}{\operatorname{soc}}
\newcommand{\depth}{\operatorname{depth}}

\makeatletter
\newcommand*{\bt}{}
\DeclareRobustCommand*{\bt}{%
  {\mathbin{\mathpalette\bt@{}}}%
}
\newcommand*{\bt@scalefactor}{.75}
\newcommand*{\bt@widthfactor}{1.4}
\newcommand*{\bt@}[2]{%
  \sbox0{$#1\vcenter{}$}%
  \sbox2{$#1\cdot\m@th$}%
  \hbox to \bt@widthfactor\wd2{%
    \hfil
    \raise\ht0\hbox{%
      \scalebox{\bt@scalefactor}{%
        \lower\ht0\hbox{$#1\bullet\m@th$}%
      }%
    }%
    \hfil
  }\hspace{.07em}
}
\makeatother

\newcommand{\enger}{\setlength{\arraycolsep}{.8pt}
                           \renewcommand{\arraystretch}{0.7} }
\newcommand{\tikzenger}{\setlength{\arraycolsep}{1pt}
                           \renewcommand{\arraystretch}{0.65} }
\newcommand{\matze}[2]{\enger{
\left(
\begin{array}{c}
\scriptstyle #1 \\
\scriptstyle #2 \\
\end{array}
\right)
}}
\newcommand{\matez}[2]{\enger{
\left(
\begin{array}{rr}
\scriptstyle #1\; & \; \scriptstyle #2\; \\
\end{array}
\right)
}}
\newcommand{\tikzmatze}[2]{\tikzenger{
\left(
\begin{array}{c}
#1 \\
#2 \\
\end{array}
\right)
}}
\newcommand{\tikzmatez}[2]{\tikzenger{
\left(
\begin{array}{rr}
#1\; & \;#2\; \\
\end{array}
\right)
}}
\newcommand{\tikzmatde}[3]{\tikzenger{
\left(
\begin{array}{c}
#1 \\
#2 \\
#3 \\
\end{array}
\right)
}}
\newcommand{\tikzmatdd}[9]{\tikzenger{
\left(
\begin{array}{cccc}
#1 & #2 & #3 \\
#4 & #5 & #6 \\
#7 & #8 & #9 \\
\end{array}
\right)
}}
\newcommand{\tikzmatzd}[6]{\tikzenger{
\left(
\begin{array}{ccc}
#1 & #2 & #3 \\
#4 & #5 & #6 \\
\end{array}
\right)
}}

\title{On stable equivalences, perfect exact sequences and Gorenstein-projective modules}


\author{Sebastian Nitsche}
\address{Sebastian Nitsche, Institute of Algebra and Number Theory, University of Stuttgart,\; Pfaffenwaldring 57, 70569 Stuttgart, Germany}
\email{sebastian.nitsche@mathematik.uni-stuttgart.de}

\begin{document}

\begin{abstract}
We consider the equivalence from the stable module category to a subcategory $\LL_A$ of the 
homotopy category constructed by Kato. This equivalence induces a correspondence between 
distinguished triangles in the homotopy category and perfect exact sequences in the module category. 
We show that an exact equivalence between categories $\LL_A$ and $\LL_B$ induces a stable equivalence of Morita type
between two finite dimensional algebra $A$ and $B$ under a separability assumption.
Moreover, we provide further sufficient conditions for a stable equivalence induced by an exact functor
to be of Morita type. This is shown using perfect exact sequences.
In particular, we study when a stable equivalence preserves perfect exact sequences up to projective direct summands.
As an application, we show that a stable equivalence preserves the category of stable Gorenstein-projective modules
if it preserves perfect exact sequences.
\end{abstract}

\maketitle

\section*{Introduction}
Besides Morita equivalences and derived equivalences, stable equivalences provide another important link between
the representation theory of two finite dimensional algebras $A$ and $B$.
In contrast to the other two, a general stable equivalence often preserves relatively few properties.
Moreover, a stable equivalence does not need to be induced by a functor on the level of the module category.
To remedy this, a stronger class of stable equivalences, called stable equivalences of Morita type, 
has been introduced by Brou\'e.
Stable equivalences of Morita type have been shown to preserve many properties of the algebra.

A result by Rickard in \cite{Rickard_RecentAdvances} states that for self-injective algebras
every stable equivalence that is induced by an exact functor between the module categories 
is isomorphic to a stable equivalence of Morita type. 
This has been generalized by Dugas and Mart\'inez-Villa as follows
in \cite{DugasMartinezVilla_MoritaType} for algebras whose semisimple quotient is separable.
A stable equivalence induced by tensoring with a bimodule ${}_AM_B$ which is projective on both sides 
is of Morita type if $\Hom_A(M,A)$ is projective over $B$. 
We apply their result to the following setting.

In \cite{Kato_Kernels}, Kato gives an equivalence $\stmod A \to \LL_A$ 
between the stable module category and a full subcategory $\LL_A$
of the unbounded homotopy category of projective modules.
Therefore, any stable equivalence induces an equivalence between $\LL_A$ and $\LL_B$.
However, if the equivalence is induced by a bimodule $M$,
we obtain that $M$ and $\Hom_B(M,B)$ induce a stable equivalence of Morita type.
\begin{LetterTheorem}[\cref{Thm:EquivOnLL-MoritaType}] \label{IntThm:EquivOnL_MoritaType}
Let $A$ and $B$ be finite dimensional $k$-algebras whose semisimple quotients are separable.

Suppose given a bimodule ${}_A M_B$ such that applying $-\otimes_A M$ componentwise induces an equivalence 
$\LL_A \xrightarrow{\sim} \LL_B$.
Let ${}_BN_A := \Hom_B(M,B)$. Then $M$ and $N$ induce a stable equivalence of Morita type between $A$ and $B$.
\end{LetterTheorem}
Moreover, every stable equivalence of Morita type arises in this way.
If $M$ and $N$ induce a stable equivalence of Morita type such that both have no projective bimodule as a direct summand,
then $- \otimes_A M$ induces an equivalence $\LL_A \to \LL_B$; cf.\ \cite[Theorem 6.5]{Nitsche_TriangulatedHull}.

As an application of \cref{IntThm:EquivOnL_MoritaType}, we give several new sufficient conditions, 
for when a stable equivalence induced by an exact functor is of Morita type. 
\begin{LetterTheorem}[\cref{Thm:StableEquivMoritaType-Overview}] \label{IntThm:ExactEquiv_MoritaType}
Let $A$ and $B$ be finite dimensional algebras whose semisimple quotients are separable.
Suppose given a bimodule $M$ which is projective as left $A$- and as right $B$-module such that 
$-\otimes_A M$ induces a stable equivalence $\stmod A \to \stmod B$.

Suppose one of the following conditions holds.
           \begin{itemize}           
           \item[$(i)$] The homology $\HH_k((F^\bt \otimes_A M)^\ast)$ vanishes for all $F^\bt \in \LL_A$ and $k \geqslant 0$.
           
           \item[$(ii)$] There exist natural isomorphisms $\nu_B(P \otimes_A M) \simeq \nu_A(P) \otimes_A M$ 
                         for all $P \in \proj A$.
                        
          \item[$(iii)$] There exists a natural isomorphism $M \otimes_B \Du\!B \simeq \Du\!A \otimes_A M$
                         of right $B$-modules.
           
           \item[$(iv)$] The algebras $A$ and $B$ have no nodes.
                        At least one of $A$ or $B$ has dominant dimension at least $1$ and finite representation type.
                        Moreover, for all simple $A$-modules $S$ whose injective hull is not projective, the image
                        $S \otimes_A M$ is an indecomposable $B$-module.
                        
           \item[$(v)$] The algebras $A$ and $B$ have dominant dimension at least $1$.
                        There is a bimodule ${}_B L_A$ which is projective as left $B$- 
                        and right $A$-module and which induces the inverse stable equivalence.
                        Moreover, for all simple $A$-modules $S$ whose injective hull is not projective, the image
                        $S \otimes_A M$ is an indecomposable $B$-module.
           \end{itemize}
Then $M$ and $\Hom_B(M,B)$ induce a stable equivalence of Morita type between $A$ and $B$.
\end{LetterTheorem}
Since a stable equivalence of Morita type is induced by an exact functor between the module categories,
it preserves short exact sequences. However, it even preserves a stronger class of short exact sequences,
called perfect exact sequences; cf.\ \cref{Lem:StableEquivMoritaType_PerfSeq}.
Thereby, a short exact sequence is said to be perfect exact if the induced sequence under the 
functor $\Hom_A(-,A)$ is exact as well. 
The equivalence $\stmod A \to \LL_A$ introduced by Kato provides a correspondence between the distinguished triangles in 
$\KKp A$ and the perfect exact sequences in $\rmod A$; cf.\ \cite[Proposition 3.6]{Kato_Mono}. 
In particular, the shift of a complex in $\LL_A$ is again an element of $\LL_A$ if and only if there exists a
perfect exact sequence with projective middle term starting in the corresponding module.
The main idea for the proof of \cref{IntThm:ExactEquiv_MoritaType} is to study when a stable equivalence 
preserves perfect exact sequences up to projective direct summands and then apply \cref{IntThm:EquivOnL_MoritaType}.

Moreover, the equivalence $\stmod A \to \LL_A$ restricts to the known triangulated equivalence
between the category of stable Gorenstein-projective modules $\stgproj A$ and the category of totally acyclic modules;
cf.\ \cref{Lem:InducedEquiv_GorensteinKtac}. We have the following main result about perfect exact sequences.
\begin{LetterTheorem}[\cref{Cor:StableEquivPreservePerfSeq}, \cref{Thm:InducedEquivOnGorensteinProj}]
 \label{IntThm:StEquiv_PreservePerfSeq}
Let $\alpha : \stmod A \to \stmod B$ be a stable equivalence between finite dimensional algebras without nodes.
Consider the following conditions.
\begin{itemize}
\item[(1)] Let $0 \rightarrow X \xrightarrow{f} Y \oplus P \xrightarrow{g} Z \rightarrow 0$ be a perfect exact sequence
in $\rmod A$ without split summands where $P \in \proj A$ and $Y$ has no projective direct summand.
Then there exists a perfect exact sequence
\[
0 \rightarrow \alpha(X) \xrightarrow{\tilde{f}} \alpha(Y) \oplus \tilde{P} \xrightarrow{\tilde{g}} \alpha(Z) \rightarrow 0
\]
in $\rmod B$ with $\tilde{P} \in \proj B$ such that $\underline{\tilde{f}} \simeq \alpha(\underline{f})$ 
and $\underline{\tilde{g}} \simeq \alpha(\underline{g})$.

\item[(2)] The equivalence $\alpha$ induces a triangulated equivalence $\stgproj A \to \stgproj B$. \rule{0pt}{7mm}
\end{itemize}
If condition (1) holds for $\alpha$ and its quasi-inverse, condition (2) holds as well.
If $A$ and $B$ have finite representation type, both conditions hold.
\end{LetterTheorem}

By excluding algebras with nodes, stable equivalences preserve most short exact sequences and almost split sequences 
up to projective direct summands; cf. \cite{AuslanderReitenVI} and \cite{MartinezVilla_Properties}.
Recall that nodes are non-projective, non-injective simple modules $S$
where the middle term of the almost split sequence starting in $S$ is projective.
The following result provides sufficient conditions for a stable equivalence to preserve perfect exact sequences.
These conditions are always satisfied for representation finite algebras without nodes, 
as in \cref{IntThm:StEquiv_PreservePerfSeq}.
We say that a morphism $f : X \to Y$ in $\rmod A$ has finite depth, if $f \not\in \rad^n(X,Y)$ for some 
$n \in \ZZ_{\geqslant 1}$. See \cref{Def:Radial-of-Cat} for more details.
\begin{LetterTheorem}[\cref{Thm:PerfectSeqUnderStableEquiv}] \label{IntThm:StEquiv-PerfSeq}
Let $\alpha : \stmod A \to \stmod B$ be a stable equivalence.

Suppose given a perfect exact sequence $0 \rightarrow X \xrightarrow{f} Y \oplus P \xrightarrow{g} Z \rightarrow 0$ 
without split summands where $X$ has no node as a direct summand, $P \in \proj A$ and $Y$ has no projective direct summand.

Suppose that $f\,p$ and $g\,\pi$ have finite depth for every projection $p$ onto an indecomposable direct summand of $Y$
and every projection $\pi$ onto an indecomposable direct summand of $Z$.
Then there exists a perfect exact sequence 
\[
0 \rightarrow \alpha(X) \xrightarrow{\tilde{f}} \alpha(Y) \oplus \tilde{P} \xrightarrow{\tilde{g}} \alpha(Z) \rightarrow 0
\]
in $\rmod B$ with $\tilde{P} \in \proj B$ such that $\underline{\tilde{f}} \simeq \alpha(\underline{f})$ 
and $\underline{\tilde{g}} \simeq \alpha(\underline{g})$.
\end{LetterTheorem}
While the assumption on the depth of $f$ and $g$ is necessary for our proof, it seems unclear whether a similar
result holds in a more general setting.

This article is structured as follows. In \cref{Sec:Preliminaries}, we recall some necessary notation and the construction
of the category $\LL_A$. In \cref{Sec:PerfectExactSequences}, we recall the definition of perfect exact sequences and
give some basic properties. In \cref{Sec:PES_StableEquivalences}, we show that certain perfect exact sequences are
preserved by a stable equivalence. This is applied in \cref{Sec:Gorenstein} to the category of stable Gorenstein-projective
modules. Finally, we consider stable equivalences of Morita type in \cref{Sec:StEquivMoritaType}.
First, we show the result of \cref{IntThm:EquivOnL_MoritaType}, that certain equivalences on the level of the 
category $\LL_A$ induce a stable equivalence of Morita type.
Afterwards, we use this result and the results of \cref{Sec:PES_StableEquivalences} to prove \cref{IntThm:ExactEquiv_MoritaType}.

\section{Preliminaries}
\label{Sec:Preliminaries}

Let $A$ and $B$ be finite dimensional $k$-algebras over a field $k$.
We assume that $A$ and $B$ do not have any semisimple direct summands.

Let $\rmod A$ be the category of finitely generated right $A$-modules.
We write $\proj A$ for the full subcategory of projective $A$-modules and
$\PP_A$ for the full subcategory of projective-injective $A$-modules.
If not specified otherwise, modules are finitely generated right modules.

Given morphisms $f : X \to Y$ and $ g: Y \to Z$ in a category, we denote the composite of $f$ and $g$ by $fg : X \to Z$.
On the other hand, we write $\GG \circ \FF$ for the composite of two functors
$\FF : \mathcal{C} \to \mathcal{D}$ and $\GG : \mathcal{D} \to \mathcal{E}$ between categories.
We write $\Du(-) := \Hom_k(-,k) : \rmod A \to A$-mod for the $k$-duality.
Recall that the functor $(-)^\ast := \Hom_A(-,A) : \rmod A \to A$-mod restricts to an equivalence
$(-)^\ast : \proj A \to A$-proj. The Nakayama functor $\nu_A : \rmod A \to \rmod A$
is given by $\nu_A := \Du\Hom_A(-,A)$.

The \textit{stable module category} $\stmod A$ is the category with the same objects as $\rmod A$
and with morphisms $\underline{\Hom}_A(X,Y) := \Hom_A(X,Y)/\mathrm{PHom}_A(X,Y)$ for $X,Y \in \rmod A$.
A morphism $f : X \to Y$ is an element of $\mathrm{PHom}_A(X,Y)$ if there exists a $P \in \proj A$ 
such that $f$ factors through $P$. We say that $A$ and $B$ are stably equivalent, if there exists an equivalence
$\stmod A \to \stmod B$.
We recall the definition of stable equivalences of Morita type.
\begin{Definition}[Brou\'e] \label{Def:StEquivMoritaType}
Let ${}_A M_B$ and ${}_B N_A$ be bimodules such that ${}_A M$, $M_B$, ${}_B N$ and $N_A$ are projective.
We say that $M$ and $N$ induce a \textit{stable equivalence of Morita type} if  
\[ 
{}_A M \otimes_B N_A \simeq A \oplus P \text{ and } {}_B N \otimes_A M_B \simeq A \oplus Q
\]
as bimodules such that ${}_A P_A$ and ${}_B Q_B$ are projective bimodules.
\end{Definition}

We introduce some notation for the category of chain complexes $\CC A$.
An element $F^\bt = (F^k)_{k \in \ZZ}$ in $\CC A$ will be written as a cochain complex with differential 
$(d^k)_{k \in \ZZ} := (d_F^k)_{k \in \ZZ}$ as follows.
\[
\cdots \rightarrow F^{-2} \xrightarrow{d^{-2}} F^{-1} \xrightarrow{d^{-1}} F^0 \xrightarrow{d^0} F^1 \xrightarrow{d^1} F^2 \rightarrow \cdots
\]
We denote the cohomology of $F^\bt$ in degree $k \in \ZZ$ by $\HH^k(F^\bt) := \Ker d^k/\im d^{k-1}$.
Truncation $\tau_{\leqslant n}\, F^\bt$ of $F^\bt$ is defined as follows for $n \in \ZZ$.
\[
\cdots \rightarrow F^{n-2} \xrightarrow{d^{n-2}} F^{n-1} \xrightarrow{d^{n-1}} F^n \rightarrow 0 \rightarrow 0 \rightarrow \cdots
\]
We often abbreviate $F^{\leqslant n} := \tau_{\leqslant n}\, F^\bt$ and similarly for 
$F^{\geqslant n} := \tau_{\geqslant n}\, F^\bt$.
The equivalence $(-)^\ast$ can be extended by componentwise application to an equivalence 
\[
(-)^\ast : \CCp A \to \CCpleft A : F^\bt \to F_\bt^\ast = F^{\bt, \ast}.
\]
Here, we write $(F_k^\ast)_{k \in \ZZ} := (F^{k,\ast})_{k \in \ZZ} := \left((F^k)^\ast\right)_{k \in \ZZ}$ 
as the chain complex with the differentials given by
$d^{F^\ast}_k := d_F^{k,\ast} := \left(d_F^k\right)^\ast$ for $k \in \ZZ$.
\[
\cdots \rightarrow F_2^\ast \xrightarrow{d_1^\ast} F_1^\ast \xrightarrow{d_0^\ast} F_0^\ast
		 \xrightarrow{d_{-1}^\ast} F_{-1}^\ast \xrightarrow{d_{-2}^\ast} F_{-2}^\ast \rightarrow \cdots
\]
We denote the homology of $F_\bt^\ast$ in degree $k \in \ZZ$ by $\HH_k(F_\bt^\ast) = \ker(d_{k-1}^\ast)/\im(d_k^\ast)$.
In this sense, we use both chain complexes and cochain complexes in our notation.
However, we reserve the notation of chain complexes for dualized cochain complexes.

Let $\KKp A$ be the unbounded homotopy category of projective modules.
A complex $F^\bt \in \KKp A$ is said to be \textit{totally acyclic,} 
if $\HH^k(F^\bt) = 0$ and $\HH_k(F_\bt^\ast) = 0$ for all $k \in \ZZ$.
The full subcategory of totally acyclic complexes in $\KKp A$ is denoted by $\KKptac A$.

In \cite{Kato_Kernels}, Kato constructs an equivalence $\stmod A \to \LL_A$ between the stable module category and
a full subcategory $\LL_A$ of $\KKp A$.
\begin{Definition}[Kato] \label{Def:L-Category}
Let $\LL_A$ be the full subcategory of $\KKp{A}$ defined as follows.
\[
\LL_A = \{ F^\bt \in \KKp{A} \;\vert\; \mathrm{H}^{<0}(F^\bt) = 0, \, \mathrm{H}_{\geqslant0}(F^\ast_\bt) = 0 \}
\]
\end{Definition}

Suppose given $X \in \rmod A$. We sketch the construction of a complex $F_X^{\bt} \in \LL_A$.
See \cite[Lemma 2.9]{Kato_Mono} for more details.

We first let $\tleq F_X^\bt$ be the minimal projective resolution of $X$.
Next, we complete $F_X^{0,\ast} \to F_X^{-1,\ast}$ to the minimal projective resolution $\tau_{\geqslant -1} F_X^{\bt, \ast}$ 
of $\Tr(X)$. Splicing these two projective resolutions together gives a complex $F_X^{\bt} \in \LL_A$. 
In particular, $\tau_{\geqslant 1} F_X^{\bt, \ast}$ is the minimal projective resolution of $X^\ast$. 
Note that $F_X^0$ is the projective cover of $X$ and $F_X^{1,\ast}$ the projective cover of $X^\ast$.
If $X$ is simple, $\nu(F_X^0)$ is the injective hull of $X$.

For $X \xrightarrow{f} Y$ in $\rmod A$, we lift $f$ to a morphism between projective resolutions 
$\tleq F_X^\bt \to \tleq F_Y^\bt$ in non-negative degrees. Similarly, we lift $f^\ast$ to a morphism
$\tau_{\geqslant -1} F_X^{\bt, \ast} \to \tau_{\geqslant -1} F_X^{\bt, \ast}$ such that the lifts coincide in degree $-1$
and $0$. Together, this gives a a morphism $f^\bt : F_X^\bt \to F_Y^\bt$ in $\LL_A$.

In the setting of commutative rings, it was shown in \cite[Theorem 2.6]{Kato_Kernels} 
that this defines an equivalence $\stmod A \to \LL_A$. However, the same arguments given there work 
for arbitrary finite dimensional algebras.
In the future, we will often use this equivalence without further comment.
\begin{Theorem}[Kato] \label{Thm:Equivalence_F}
The mapping $X \mapsto F_X^\bt$ defines a functor $\FF : \stmod A \to \KKp A$. 
The functor $\FF$ restricts to an equivalence
\[
\FF : \stmod A \xrightarrow{\sim} \LL_A
\]
with quasi-inverse $\HH^0(\tleq(-)) : \LL_A \to \stmod A$.
\end{Theorem}

Recall that $\PP_A$ denotes the full subcategory of projective-injective $A$-modules.
We write $\Pperp_A$ for the full subcategory of $\rmod A$ with objects $X \in \rmod A$ such that
$\Hom_A(X,Z) = 0$ for all $Z \in \PP_A$.

\begin{Lemma} \label{Lem:L_in_Pperp}
Let $F^\bt \in \LL_A$. We have $\HH^k(F^\bt) \in \Pperp_A$ for all $k \in \ZZ$.
\end{Lemma}
\begin{proof}
Let $k \in \ZZ$. If $k < 0$, we have $\HH^k(F^\bt) = 0$ and there is nothing to show. Thus, we assume that $k \geqslant 0$.
Suppose given $Z \in \PP_A$ and $f : \HH^k(F^\bt) \to Z$.
Consider the following commutative diagram. The morphism $\alpha$ exists since $Z$ is injective. 
\[
\begin{tikzcd}[row sep=.9cm, column sep=.7cm]
\cdots \ar[r] & F^{k-1} \ar[r, "d^{k-1}"] &  F^k \ar[r] \ar[dr, twoheadrightarrow, "\pi"] & F^{k+1} \ar[r] & \cdots \\
 & \Ker{d^{k}} \ar[ur, hookrightarrow] \ar[dr, twoheadrightarrow] & & F^k /\, \im{d^{k-1}} \ar[ddl, dashed, "\alpha"]\\
 & & \HH^k(F^\bt) \ar[ur, hookrightarrow] \ar[d, "f"] \\
 & & Z
\end{tikzcd} 
\]
Since $d^{k-1} \, \pi \, \alpha = 0$, this yields a morphism of complexes $\varphi^\bt := \pi \, \alpha : F^\bt \to Z[-k]$
if we consider $Z \in \KKp A$ as the complex concentrated in degree zero.
Applying $(-)^\ast = \Hom_A(-,A)$, we obtain a morphism of complexes $\varphi^\ast_\bt : Z^\ast[-k] \to F_\bt^\ast$.
Since $F^\bt \in \LL_A$ and $k \geqslant 0$, we have $\HH_k(F_\bt^\ast) = 0$. Using that $Z$ is projective, we obtain $\varphi_\bt^\ast = 0$ and
thus also $\varphi^\bt = 0$. In particular, we have $\varphi^k = \pi\,\alpha = 0$. Since $\pi$ is surjective, this results
in $f = 0$.
\end{proof}

\section{Perfect exact sequences}
\label{Sec:PerfectExactSequences}

The following class of short exact sequences will be our main tool throughout this article.
\begin{Definition} \label{Def:PerfSeq}
A short exact sequence $0 \to X \to Y \to Z \to 0$ in $\rmod A$ is called \textit{perfect exact} if
the induced sequence $0 \to X^\ast \to Y^\ast \to Z^\ast \to 0$ is exact in $A$-mod.
\end{Definition}

Since $\Hom_A(-,A)$ is right exact, a short exact sequence $0 \to X \xrightarrow{f} Y \to Z \to 0$ in $\rmod A$ is perfect exact 
if and only if $\Hom_A(f,A)$ is surjective. In particular, this holds if $X^\ast = 0$.
Another class of examples is given by almost split sequences with non-projective starting term.
Throughout, we will denote the almost split sequence starting in an indecomposable 
non-injective $A$-module $X$ as follows.
\[ 
0 \rightarrow X \rightarrow E_X \rightarrow \tau^{-1}(X) \rightarrow 0 
\]
\begin{Lemma} \label{Ex:PerfectSeq}
An almost split sequence $0 \to X \xrightarrow{f} E_X \xrightarrow{g} \tau^{-1}(X) \to 0$ 
is perfect exact if and only if $X$ is not projective. 
\end{Lemma}
\begin{proof}
A morphism $p : X \to A$ induces a morphism $\beta : E_X \to A$ with 
$f\, \beta = p$ if and only if $p$ is not split. However, this holds if and only if $X$ is not projective,
since the starting term of  an almost split sequence is indecomposable.
\end{proof}

Finally, we will need the following sufficient condition for algebras with dominant dimension at least one.
\begin{Lemma} \label{Lem:Pperp_PerfSeq_IfDomDim}
Suppose that $\ddim A \geqslant 1$. Let $Y \in \Pperp_A$.
Let $Y'$ be a submodule of $Y$.

We have $(Y')^\ast = 0$ and every short exact sequence $0 \to X \to Y' \to Z \to 0$ in $\rmod A$
is a perfect exact sequence. 
\end{Lemma}
\begin{proof}
Let $Y'$ be a submodule of $Y$.
Since the embedding $Y' \hookrightarrow Y$ is injective, the condition $Y \in \Pperp$ 
implies that $Y'$ is contained in $\Pperp$ as well.
Using that $\ddim A \geqslant 1$, we obtain that $\Hom(Y',A) = 0$.

Suppose given a short exact sequence $0 \to X \to Y' \to Z \to 0$ in  $\rmod A$.
Since $X$ is isomorphic to a submodule of $Y$, we have seen above that $X^\ast = 0$.
Thus, $0 \to X \to Y' \to Z \to 0$ is a perfect exact sequence.
\end{proof}

The equivalence $\FF: \stmod A \to \LL_A$ induces a correspondence between perfect exact sequences in $\mod A$
and distinguished triangles in $\KKp A$. The following is based on \cite[Proposition 3.6]{Kato_Mono}.
\begin{Proposition}[Kato] \label{Prop:PerfSeq-DistTriang}
Suppose given a sequence $X \xrightarrow{f} Y \xrightarrow{g} Z$ in $\rmod A$ such that 
$Y$ has no projective direct summand.
The following are equivalent.
\begin{itemize}
\item[(1)] There exists a projective module $P$ and morphisms $p$ and $q$ such that 
\[ 
0 \rightarrow X \xrightarrow{\matez{f}{p}} Y \oplus P \xrightarrow{\matze{g}{q}} Z \rightarrow 0
\]
is a perfect exact sequence in $\rmod A$.

\item[(2)] The sequence 
$ F_X^\bt \xrightarrow{f^\bt} F_Y^\bt \xrightarrow{g^\bt} F_Z^\bt \to$
is a distinguished triangle in $\KKp A$.
\end{itemize}
\end{Proposition}

Perfect exact sequences with projective middle term are of special importance for us.
They correspond to the shift in $\KKp A$ under the equivalence $\FF$ of \cref{Thm:Equivalence_F}.
\begin{Lemma} \label{Lem:Ext=0_PerfectSeq}
The following are equivalent for $Z \in \rmod A$.
\begin{itemize}
\item[(1)] $\Ext_A^1(Z,A) = 0$

\item[(2)] There exists a perfect exact sequence $0 \to X \to P \to Z \to 0$ with $P \in \proj A$.

\item[(3)] Every short exact sequence ending in $Z$ is perfect exact.

\item[(4)] $F_Z^\bt[-1] \in \LL_A$
\end{itemize}
\end{Lemma}
\begin{proof}
Suppose given a short exact sequences $0 \to X \to Y \to Z \to 0$ in $\rmod A$.
Applying $\Hom_A(-,A)$, we obtain the following exact sequence.
\[
0 \to \Hom_A(Z,A) \to \Hom_A(Y,A) \to \Hom_A(X,A) \to \Ext^1_A(Z,A) \to \Ext_A^1(Y,A)
\]
Thus, the short exact sequence is perfect exact if $\Ext^1_A(Z,A) = 0$.
If in addition $Y$ is projective, then the converse holds as well.
Note that there always exists a short exact sequence of the form $0 \to X \to P \to Z \to 0$ with $P$ the 
projective cover of $Z$.

We verify the equivalence of $(2)$ and $(4)$.
Consider the following distinguished triangle in $\KKp A$.
\[
F_Z^\bt[-1] \to 0 \to F_Z^\bt \to 
\]
By \cref{Prop:PerfSeq-DistTriang}, $F_Z^\bt[-1] \in \LL_A$ if and only if we have a perfect exact sequence
\[
0 \to X \to P \to Z \to 0
\]
in $\rmod A$ with $P \in \proj A$ and $X:= \HH^0(\tleq F_Z^\bt[-1])$. 
In this case, $F_Z^\bt[-1] \simeq F_X^\bt \in \LL_A$.
\end{proof}

Now, we consider the starting term of a perfect exact sequence with projective middle term.
We are mainly interested in the case of simple modules. 
Recall that $F_X^0$ is the projective cover of $X \in \rmod A$.
If $F_X^1 = 0$, the complex $F_X^\bt$ is the minimal projective resolution of $X$.

\begin{Lemma} \label{Lem:Charact_nu(S)=0}
The following are equivalent for $X \in \rmod A$.
\begin{itemize}
\item[(1)] $F_X^\bt[1] \in \LL_A$.

\item[(2)] There exists a perfect exact sequence $0 \to X \to P \to Z \to 0$ in $\rmod A$ with $P \in \proj A$.
\end{itemize}
Furthermore, (1) and (2) imply the following equivalent conditions.
\begin{itemize}
\item[(3)] $X^\ast \neq 0$.

\item[(4)] $F_X^1 \neq 0$.
\end{itemize}
If $X$ is simple, all four conditions are equivalent.
\end{Lemma}
\begin{proof}
Suppose that $F_X^\bt[1] \in \LL_A$. By \cref{Prop:PerfSeq-DistTriang}, the distinguished triangle
\[
F_X^\bt \to 0 \to F_X^\bt[1] \to
\]
induces a perfect exact sequence $0 \to X \to P \to Z \to 0$ with $Z := \HH^0\big(\tleq (F^\bt[1])\big)$ and $P \in \proj A$.
On the other hand, any perfect exact sequence $0 \to X \to P \to Z \to 0$ with $P \in \proj A$ induces 
the following distinguished triangle in $\LL_A$.
\[
F_X^\bt \to 0 \to F_Z^\bt \to 
\]
In this case, $F_X^\bt[1] \simeq F_Z^\bt \in \LL_A$.
This shows the equivalence of (1) and (2).
There is nothing to show for the implication (2) $\Rightarrow$ (3).
Using that $\tau_{\geqslant 1} F_X^{\bt, \ast}$ is the minimal projective resolution of $X^\ast$, we obtain
the equivalence of (3) and (4).

Now, suppose that that $X$ is simple and $F_X^1 \neq 0$.
We have $\HH^0(F_X^\bt) \subseteq \Cok(F_X^{-1} \to F_X^0) \simeq X$. 
Since $F_X^1 \neq 0$, we know that $\Ker(F_X^0 \to F_X^1) \neq F_X^0$ so that the above inclusion is not trivial.
If $X$ is simple, we obtain $\HH^0(F_X^\bt) = 0$. This implies $F_X^\bt[1] \in \LL_A$.
\end{proof}

Finally, we give two methods to construct perfect exact sequences from existing ones.
We start with a construction via pushout and pullback which will be used to merge a perfect exact sequence
with an almost split sequence.
\begin{Lemma}\label{Lem:PerfectSeq-Via-PoPb}
Suppose given a  short exact sequence $0 \to X \xrightarrow{f} Y \xrightarrow{g} Z \to 0$ in $\rmod A$.
\begin{itemize}
\item[(1)] Let $0 \to X \xrightarrow{u} U \xrightarrow{v} V \to 0$ be a short exact sequence in $\rmod A$
           such that $f = u \, \alpha$ via a morphism $\alpha : U \to Y$. 
           Then there exists a short exact sequence such that the following diagram commutes.
           \[
            \begin{tikzcd}[row sep=1cm, column sep=1cm]
            0 \ar[r] & X \ar[r, "f"] \ar[d, "u"]         & Y \ar[r, "g"] \ar[d, "\tikzmatez{1}{0}"]     & Z \ar[r] \ar[d, equal]     & 0 \\
            0 \ar[r] & U \ar[r, "\tikzmatez{\alpha}{v}"] & Y \oplus V \ar[r, "\tikzmatze{g}{\beta}"]    & Z \ar[r]                   & 0 
            \end{tikzcd}                       
           \]
           If $0 \to X \xrightarrow{u} U \xrightarrow{v} V \to 0$ and the upper row of this diagram are perfect exact, 
           then so is the lower row.
\item[(2)] Let $0 \to U \xrightarrow{u} V \xrightarrow{v} Z \to 0$ be a short exact sequence in $\rmod A$
           such that $g = \alpha \, v$ via a morphism $\alpha : Y \to V$. 
           Then there exists a short exact sequence such that the following diagram commutes.
           \[
            \begin{tikzcd}[row sep=1cm, column sep=1cm]
            0 \ar[r] & X \ar[r, "\tikzmatez{\beta}{f}"] \ar[d, equal]   & U \oplus Y \ar[r, "\tikzmatze{u}{\alpha}"] \ar[d, "\tikzmatze{0}{1}"] & V \ar[r] \ar[d, "v"] & 0 \\
            0 \ar[r] & X \ar[r, "f"]                                    & Y \ar[r, "g"]                                                         & Z \ar[r]             & 0
            \end{tikzcd}                       
           \]
           If $0 \to U \xrightarrow{u} V \xrightarrow{v} Z \to 0$ and the lower row of this diagram are perfect exact, 
           then so is the upper row.
\end{itemize}
\end{Lemma}
\begin{proof}
First, we only proof the existence of a short exact sequence and afterwards check that this is a perfect exact sequence.
In both cases we show part (1) since part (2) follows dually.
We start by verifying that the following is a pushout-square.
\[
\begin{tikzcd}
X \ar[r, "f"] \ar[d, "u"]   & Y \ar[d, "\tikzmatez{1}{0}"] \\
U \ar[r, "\tikzmatez{\alpha}{v}"] & Y \oplus V
\end{tikzcd}
\]
By assumption, this diagram commutes. Suppose given $T \in \rmod A$ together with morphism $t_1 : Y \to T$ and $t_2 : U \to T$
such that $f \, t_1 = u \, t_2$.
We construct $\varphi : Y \oplus V \to T$ such that the following diagram commutes.
\[
\begin{tikzcd}[row sep=1cm, column sep=1cm]
X \ar[r, "f"] \ar[d, "u"]                                     & Y \ar[d, "\tikzmatez{1}{0}"] \ar[ddr, bend left ," t_1"] &   \\
U \ar[r, "\tikzmatez{\alpha}{v}"] \ar[drr, bend right, "t_2"] & Y \oplus V \ar[dr, dashed, "\varphi"]                    &   \\
                                                              &                                                          & T
\end{tikzcd}
\]
We have $u (t_2 - \alpha \, t_1) = u \, t_2 - f \, t_1 = 0$. Hence, there exists a unique $\beta : V \to T$ such that 
$v \, \beta = t_2 - \alpha \, t_1$.
\[
\begin{tikzcd}[column sep=1.5cm]
X \ar[r, "u"] & U \ar[r, "v"] \ar[d, "t_2 - \alpha \, t_1"] & V \ar[dl, dashed, bend left, "\beta"] \\
              & T                                           &
\end{tikzcd}
\]
This yields a unique $\varphi = \matze{t_1}{\beta}$ such that the diagram above commutes.
In conclusion, letting $T:= Z$, the pushout-square induces the following commutative diagram with exact rows.
 \[
            \begin{tikzcd}[row sep=1cm, column sep=1cm]
            0 \ar[r] & X \ar[r, "f"] \ar[d, "u"]         & Y \ar[r, "g"] \ar[d, "\tikzmatez{1}{0}"]     & Z \ar[r] \ar[d, equal]     & 0 \\
            0 \ar[r] & U \ar[r, "\tikzmatez{\alpha}{v}"] & Y \oplus V \ar[r, "\tikzmatze{g}{\beta}"]    & Z \ar[r]                   & 0 
            \end{tikzcd}                       
 \]
\textit{Perfect exact sequence.} Suppose that $0 \to X \xrightarrow{f} Y \xrightarrow{g} Z \to 0$ is a perfect exact sequence. 
That is,
$0 \to Z^\ast \xrightarrow{g^\ast} Y \xrightarrow{f^\ast} X^\ast \to 0$ is exact in $A$-mod.
If additionally, $0 \to V^\ast \xrightarrow{v^\ast} U^\ast \xrightarrow{u^\ast} X^\ast \to 0$ is exact,
we can apply (2) for left $A$-modules to obtain the following exact sequence.
\[
0 \to Z^\ast \xrightarrow{\matez{\beta^\ast}{g^\ast}} V^\ast \oplus Y^\ast \xrightarrow{\matze{v^\ast}{\alpha^\ast}} U^\ast \to 0
\]
Hence, $0 \to U \xrightarrow{\matez{\alpha}{v}} Y \oplus V \xrightarrow{\matze{g}{\beta}} Z \to 0$ is perfect exact.
\end{proof}

The next construction via the snake lemma will be used to reverse the process of the previous lemma.
\begin{Lemma} \label{Lem:PerfectSeq-Via-Snake}

The following holds.

\begin{itemize}
\item[(1)] Suppose given two short exact sequences in $\rmod A$ of the following form.
\begin{align*}
& 0 \to X \xrightarrow{\matez{s}{\iota}} U \oplus P \xrightarrow{\matze{t}{\pi}} V \to 0 \\
& 0 \to U \xrightarrow{\matez{v}{t}} Y \oplus V \xrightarrow{\matze{g}{w}} Z \to 0
\end{align*}
Then ${0 \to X \xrightarrow{\matez{s v}{\iota}} Y \oplus P \xrightarrow{\matze{g}{-\pi w}} Z \to 0}$
is a short exact sequence. If the given two sequences are perfect exact, then so is this sequence.

\item[(2)] Suppose given two short exact sequences in $\rmod A$ of the following form.
\begin{align*}
& 0 \to U \xrightarrow{\matez{s}{\iota}} V \oplus P \xrightarrow{\matze{t}{\pi}} Z \to 0 \\
& 0 \to X \xrightarrow{\matez{f}{v}} Y \oplus U \xrightarrow{\matze{w}{s}} V \to 0
\end{align*}
Then ${0 \to X \xrightarrow{\matez{f}{-v \iota}} Y \oplus P \xrightarrow{\matze{w t}{\pi}} Z \to 0}$
is a short exact sequence. If the given two sequences are perfect exact, then so is this sequence.
\end{itemize}
\end{Lemma}
\begin{proof}
First, we only proof the existence of a short exact sequence and afterwards check that this is a perfect exact sequence.
In both cases we show part (1) since part (2) follows dually.
Note that there is an isomorphism of sequences
\[
\begin{tikzcd}[row sep=1.4cm, column sep=1.5cm]
0 \ar[r] & U \oplus P \ar[r, "\tikzmatzd{v}{t}{0}{0}{0}{1}"]   \ar[d, equal] & Y \oplus V \oplus P \ar[r, "\tikzmatde{g}{w}{0}"]       \ar[d, "\sim"' sloped,"\tikzmatdd{1}{0}{0}{0}{1}{0}{0}{\pi}{1}"] & Z \ar[r] \ar[d, equal] & 0 \\
0 \ar[r] & U \oplus P \ar[r, "\tikzmatzd{v}{t}{0}{0}{\pi}{1}"]               & Y \oplus V \oplus P \ar[r, "\tikzmatde{g}{w}{-\pi\,w}"]                                                                  & Z \ar[r]               & 0
\end{tikzcd}
\]
so that the lower sequence is exact as well.
Consider the following commutative diagram.
\[
\begin{tikzcd}[row sep=1.3cm, column sep=1cm, ampersand replacement=\&]
              \& 0 \ar[r]                                           \ar[d]        \& 0                   \ar[r]                        \ar[d]                                   \& X                 \ar[d, "\tikzmatez{s}{\iota}"] \&   \\    
0 \ar[r]      \& 0 \ar[r]                                           \ar[d]        \& U \oplus P          \ar[r, equal]                 \ar[d, "\tikzmatzd{v}{t}{0}{0}{\pi}{1}"] \& U \oplus P \ar[r] \ar[d, "\tikzmatze{t}{\pi}"]   \& 0 \\
0 \ar[r]      \& Y \oplus P \ar[r, "\tikzmatzd{1}{0}{0}{0}{0}{1}"]  \ar[d, equal] \& Y \oplus V \oplus P \ar[r, "\tikzmatde{0}{1}{0}"] \ar[d, "\tikzmatde{g}{w}{-\pi\,w}"]      \& V          \ar[r] \ar[d]                         \& 0 \\
              \& Y \oplus P \ar[r, "\tikzmatze{g}{-\pi\,w}"]                      \& Z                   \ar[r]                                                                 \& 0                                                \& 
\end{tikzcd}
\]
The snake lemma yields a short exact sequence 
\[
0 \to X \xrightarrow{\matez{s\,v}{\iota}} Y \oplus P \xrightarrow{\matze{g}{-\pi\,w}} Z \to 0.
\]
\textit{Perfect exact sequence.} Suppose that in part $(1)$ the sequences 
\begin{align*}
& 0 \to V^\ast \xrightarrow{\matez{t^\ast}{\pi^\ast}} U^\ast \oplus P^\ast \xrightarrow{\matze{s^\ast}{\iota^\ast}} X^\ast \to 0 \\
& 0 \to Z^\ast \xrightarrow{\matez{g^\ast}{w^\ast}} Y^\ast \oplus V^\ast \xrightarrow{\matze{v^\ast}{t^\ast}} U^\ast \to 0
\end{align*}
are exact. We can apply part $(2)$ for left $A$-modules to obtain the short exact sequence 
\[
0 \to Z^\ast \xrightarrow{\matez{g^\ast}{-w^\ast\pi^\ast}} Y^\ast \oplus P^\ast \xrightarrow{\matze{v^\ast s^\ast}{\iota^\ast}} X^\ast \to 0.
\]
Hence, the sequence $0 \to X \to Y \oplus P \to Z \to 0$ of part $(1)$ is perfect exact.
\end{proof}

\section{Perfect exact sequences and stable equivalences}
\label{Sec:PES_StableEquivalences}

This section is dedicated to examine what happens to perfect exact sequences under stable
equivalences $\stmod A \to \stmod B$.
Being perfect exact can be seen as a property of a given sequence in $\rmod A$.
For later use, we introduce the following shortened notion for stable equivalences that preserve this property.
\begin{Definition} \label{Def:PreservePerfSeq_StableEquiv}
Let $\eta : 0 \rightarrow X \xrightarrow{f} Y \oplus P \xrightarrow{g} Z \rightarrow 0$ be a perfect exact sequence in $\rmod A$
without split summands such that $P$ is projective and $Y$ has no projective direct summand.
We say that a functor $\alpha : \stmod A \to \stmod B$ \textit{preserves the perfect exact sequence} $\eta$
if there exists a perfect exact sequence
\[
0 \rightarrow \alpha(X) \xrightarrow{\tilde{f}} \alpha(Y) \oplus \tilde{P} \xrightarrow{\tilde{g}} \alpha(Z) \rightarrow 0
\]
in $\rmod B$ with $\tilde{P} \in \proj B$ such that $\underline{\tilde{f}} \simeq \alpha(\underline{f})$ 
and $\underline{\tilde{g}} \simeq \alpha(\underline{g})$ in $\stmod B$.
\end{Definition}

For now, we can show that a stable equivalence preserves perfect exact sequences
with projective middle term if the stable equivalence and its quasi-inverse are induced by an exact functor.
\begin{Proposition} \label{Prop:ExactsStableEquiv_ProjMiddleTerm}
Let ${}_A M_B$ be a bimodule which is projective as left $A$- and right $B$-module
such that $- \otimes_A M : \rmod A \to \rmod B$ is an exact functor which induces a stable equivalence $\stmod A \to \stmod B$.
\begin{itemize}
\item[(1)] The following are equivalent.
\begin{itemize}
\item[(i)] The functor $- \otimes_A M$ preserves perfect exact sequences with projective middle term.
            
\item[(ii)] For all $Z \in \rmod A$ we have $\Ext^1_B(Z \otimes_A M, B) = 0$ if $\Ext^1_A(Z,A) = 0$.
\end{itemize}

\item[(2)] If there is a bimodule ${}_B L_A$ which is projective as left $B$- and right $A$-module and which induces the inverse stable
equivalence, then the equivalent conditions of part (1) hold.
\end{itemize}
\end{Proposition}
\begin{proof}
\textit{Ad (1).}
Suppose given $Z \in \rmod A$ with $\Ext^1(Z,A) = 0$.
We have seen in \cref{Lem:Ext=0_PerfectSeq} that in this case there exists a perfect exact sequence ending in $Z$ with
projective middle term. Thus, (i) implies that there also is a perfect exact sequence in $\rmod B$ with ending
term $Z \otimes_A M$ and projective middle term. Using \cref{Lem:Ext=0_PerfectSeq} again, we obtain 
$\Ext^1_B(Z \otimes_A M, B) = 0$. This shows the implication (i) $\Rightarrow$ (ii). 

On the other hand, suppose that $0 \to X \to P \to Z \to 0$ is a perfect exact sequence in $\rmod A$ with $P \in \proj A$.
Then $\Ext^1(Z,A) = 0$ so that (ii) implies $\Ext^1_B(Z \otimes_A M, B) = 0$.
Now, by \cref{Lem:Ext=0_PerfectSeq}, every short exact sequence ending in $Z \otimes_A M$ is perfect exact.
In particular, this holds for the induced short exact sequence 
$0 \to X \otimes_A M \to P \otimes_A M \to Z \otimes_A M \to 0$.
Therefore, the implication from (ii) to (i) holds as well.

\textit{Ad (2).} We show condition (ii) of part (1).
Suppose given $Z \in \rmod A$ with $\Ext^1(Z,A) = 0$. We write $Z' := Z \otimes_A M \in \rmod B$.
Let $0 \to B \to Y' \xrightarrow{g'} Z' \to 0$ be a short exact sequence in $\Ext^1_B(Z \otimes_A M, B)$.
Since $- \otimes_B L$ is exact, we obtain the following short exact sequence.
\[
0 \to B \otimes_B L \to Y' \otimes_B L \xrightarrow{g' \otimes L} Z' \otimes_B L \to 0
\]
Note that $Z' \otimes_B L = Z \otimes_A M \otimes_B L \simeq Z$ in $\stmod A$
and $B \otimes_B L_A \simeq L_A \in \proj A$. Using that $\Ext^1_A(Z,A) = 0$,
this implies that $g' \otimes_B L$ is a split epimorphism with projective kernel. 
Thus, $g' \otimes_B L$ is a stable isomorphism. 
As a consequence, $g' \otimes_B L \otimes_A M$ is a stable isomorphism as well. 
Using that  $g' \otimes_B L \otimes_A M \simeq g'$ in $\stmod B$ and that $g'$ is surjective,
we obtain that $g'$ is a split epimorphism.
In conclusion, $\Ext^1_B(Z \otimes_A M, B) = 0$.
\end{proof}

The situation is better for stable equivalences of Morita type.
\begin{Lemma} \label{Lem:StableEquivMoritaType_PerfSeq}
Suppose ${}_A M_B$ and ${}_B N_A$ are bimodules that induce a stable equivalence of Morita type 
such that $M$ and $N$ do not have any non-zero projective bimodule as direct summand.

If $0 \to X \xrightarrow{f} Y \xrightarrow{g} Z \to 0$ is a perfect exact sequence in $\rmod A$, then 
\[
0  \to X \otimes_A M \xrightarrow{f \otimes M} Y \otimes_A M \xrightarrow{g \otimes M} Z \otimes_A M \to 0
\] 
is a perfect exact sequence in $\rmod B$.
Similarly, the functor $- \otimes_B N$ maps perfect exact sequences in $\rmod B$ to perfect exact sequences in $\rmod A$.
\end{Lemma}
\begin{proof}
Let $\eta : 0 \to X \xrightarrow{f} Y \xrightarrow{g} Z \to 0$ be a perfect exact sequence in $\rmod A$.
Since $- \otimes_A M$ is an exact functor, 
$0  \to X \otimes_A M \xrightarrow{f \otimes M} Y \otimes_A M \xrightarrow{g \otimes M} Z \otimes_A M \to 0$
is a short exact sequence in $\rmod B$.
By \cite[Lemma 3.3]{DugasMartinezVilla_MoritaType} and \cite[Lemma 4.1]{ChenPanXi_MoritaType}, 
there exist natural isomorphisms $(U \otimes_A M)^\ast \simeq N \otimes_A U^\ast$ for all $U \in \rmod A$.
Thus, the following diagram commutes.
\[
\begin{tikzcd}
0 \ar[r] & (Z \otimes_A M)^\ast \ar[r, "(g \otimes M)^\ast"] \ar[d, "\sim" sloped] & (Y \otimes_A M)^\ast \ar[r, "(f \otimes M)^\ast" sloped] \ar[d, "\sim" sloped] & (X \otimes_A M)^\ast \ar[r] \ar[d, "\sim" sloped] & 0 \\
0 \ar[r] & N \otimes_A Z^\ast \ar[r, "N \otimes g^\ast"]                           & N \otimes_A Y^\ast \ar[r, "N \otimes f^\ast"]                                  & N \otimes_A X^\ast \ar[r]                         & 0
\end{tikzcd}
\]
Since $N \otimes_A -$ is an exact functor and $\eta$ perfect exact, the lower sequence is exact. 
This implies that the upper sequence is exact as well.
Consequently, 
\[
0  \to X \otimes_A M \xrightarrow{f \otimes M} Y \otimes_A M \xrightarrow{g \otimes M} Z \otimes_A M \to 0
\]
is a perfect exact sequence.
\end{proof}

For an arbitrary stable equivalence to preserve perfect exact sequences, we have to at least exclude short exact sequences 
that start with a node. We recall the definition of a node.
See \cite{MartinezVilla_l-hereditary} for more details.
\begin{Definition}\label{Def:Node}
A simple $A$-module $S$ is called a \textit{node} if it is neither projective nor injective
and the middle term $E_S$ of the almost split sequence starting in $S$ is projective.
\end{Definition}

We will use the following immediate characterization of a node; cf.\ \cite[Proposition 2.5.(a)]{AuslanderReitenVI}.
\begin{Lemma} \label{Lem:Char-Node}
Suppose given an almost split sequence with $X$ not injective.
\[
0 \to X \xrightarrow{f} E_X \xrightarrow{g} Z \to 0
\]
\begin{itemize}
\item[(1)] We have $\underline{f} = 0$ in $\stmod A$ if and only if $X$ or $E_X$ is projective.

\item[(2)] Suppose that $X$ is simple and not projective. Then $X$ is a node if and only if $\underline{f} = 0$.
\end{itemize}
\end{Lemma}

We also use that a node cannot be a direct summand of the middle term in an almost split sequence.    
\begin{Lemma} \label{Lem:NodeAsMiddleTermOfAss}
Suppose given an almost split sequence $0 \to X \xrightarrow{f} E_X \xrightarrow{g} Z \to 0$ in $\rmod A$.

The middle term $E_X$ does not have a node as a direct summand.
\end{Lemma}
\begin{proof}
Assume that $\iota : N \hookrightarrow E_X$ is is the embedding of a direct summand such that $N$ is a node.
We have an almost split sequence 
\[
0 \to N \xrightarrow{s} E_N \xrightarrow{t} \tau^{-1}(N) \to 0
\]
starting in $N$ with $E_N$ projective. 
The morphism $\iota\, g : N \to Z$ factors through $s$ via a morphism $u : E_N \to Z$, that is $\iota\, g = s\, u$.
Since $\iota\, g$ and $s$ are irreducible, we obtain that $u$ is a split monomorphism.
Thus, $E_N$ is a projective direct summand of $Z$. A contradiction.
\end{proof}

In \cite[Proposition 2.4 and 3.5]{AuslanderReitenVI}, Auslander and Reiten provide the following results
for the case of short exact sequences using functor categories. 
The first part was later generalized to a larger class of short exact sequences in \cite[Theorem 1.7]{MartinezVilla_Properties}
by Mart\'inez-Villa.
\begin{Proposition}[Auslander, Reiten]
\label{Prop:AssUnderStableEquiv}
Let $\alpha : \stmod A \to \stmod B$ be a stable equivalence.

Let $0 \to X \xrightarrow{f} Y \oplus P \xrightarrow{g} Z \to 0$ be a short exact sequence in $\rmod A$ 
without split summands such that $X$ is indecomposable, $P \in \proj A$ and $Y$ has no projective direct summand.

If $X$ is not a node and not projective, there exists a short exact sequence
\[
0 \to \alpha(X) \xrightarrow{\tilde{f}} \alpha(Y) \oplus \tilde{P} \xrightarrow{\tilde{g}} \alpha(Z) \to 0
\]
in $\rmod B$ with $\tilde{P} \in \proj B$ such that $\alpha(f) \simeq \tilde{f}$ and $\alpha(g) \simeq \tilde{g}$ in $\stmod B$.

If $0 \to X \xrightarrow{f} Y \oplus P \xrightarrow{g} Z \to 0$ is an almost split sequence in $\rmod A$, 
then we obtain an almost split sequence
$0 \to \alpha(X) \xrightarrow{\tilde{f}} \alpha(Y) \oplus \tilde{P} \xrightarrow{\tilde{g}} \alpha(Z) \to 0$ 
in $\rmod B$.
\end{Proposition}
We aim to prove a similar result and show that certain stable equivalences preserve perfect exact sequences.
However, our method follows an algorithmic approach. 

First, we construct a chain of perfect exact sequences $\eta_0 \to \eta_1 \to \cdots \to \eta_l$ in $\rmod A$
such that $\eta_l$ is a direct sum of almost split sequences.
By remembering the steps taken during this construction, we can reconstruct a perfect exact sequence in $\rmod B$ 
corresponding to the original perfect exact sequence $\eta_0$.
This is done by using almost split sequences during each step of the construction.
By the proposition above, such almost split sequences are
preserved by a stable equivalence between algebras without nodes.
In order for this chain to be finite, we need to assume some condition on the morphisms 
in the perfect exact sequence. This condition is satisfied for all morphisms if $A$ 
is of finite representation type. 

In case that the starting term of the perfect exact sequence is not indecomposable, we need the following remark.
\begin{Remark} \label{Rem:SumOfAss}
Suppose given $X \in \rmod A$ without injective direct summands. 
Let $X = \bigoplus_i X_i$ be the decomposition of $X$ into indecomposable direct summands.
				
We denote the direct sum of all almost split sequences starting in the $X_i$ as follows.
\[ 
0 \rightarrow X \xrightarrow{s} E_X \xrightarrow{t} T(X) \rightarrow 0 
\]
In particular, $E_X = \bigoplus_i E_{X_i}$ and $T(X) = \bigoplus_i	 \tau^{-1}(X_i)$.

Recall that almost split sequences with $X$ not projective are perfect exact; cf.\ \cref{Ex:PerfectSeq}.
We use \cref{Lem:PerfectSeq-Via-PoPb} to combine a perfect exact sequence with such almost split sequences.
\end{Remark}

\begin{Lemma}\label{Lem:ConstructPerfSeqViaARSeq}
Suppose given a perfect exact sequence $\eta : 0 \rightarrow X \xrightarrow{f} Y \xrightarrow{g} Z \rightarrow 0$ in $\rmod A$ 
without split summands. 

Recall from \cref{Rem:SumOfAss} the sequence $0 \rightarrow X \xrightarrow{s} E_X \xrightarrow{t} T(X) \rightarrow 0$.
Then there exists a perfect exact sequence
\[ 
\tilde{\eta} : 0 \rightarrow E_X \xrightarrow{\matez{v}{t}} Y \oplus T(X) \xrightarrow{\matze{g}{w}} Z \rightarrow 0
\]
in $\rmod A$ such that $f = s\, v$. We often denote this sequence by $\tilde{\eta}$.
\end{Lemma}
\begin{proof}
By assumption, $X$ has neither injective nor projective direct summands, otherwise the given exact 
sequence would have a split direct summand.
Hence, there exists a perfect exact sequence that is the direct sum of almost split sequences 
starting in direct summands of $X$; cf.\ \cref{Rem:SumOfAss}.
\[ 
0 \rightarrow X \xrightarrow{s} E_X \xrightarrow{t} T(X) \rightarrow 0 
\]
Moreover, $f$ is not a split morphism. 
Therefore, $f$ factors through $s$ via a morphism $E_X \xrightarrow{v} Y$, that is $f = s \, v$. 

By \cref{Lem:PerfectSeq-Via-PoPb}.(1), we obtain that 
$0 \rightarrow E_X \xrightarrow{\matez{v}{t}} Y \oplus T(X) \xrightarrow{\matze{g}{w}} Z \rightarrow 0$
is a perfect exact sequence with some morphism $w : T(X) \to Z$.
\end{proof}

\begin{Construction} \label{Constr:EtaTildeFromEta}
Suppose given a perfect exact sequence $\eta_0 : 0 \rightarrow X_0 \xrightarrow{f_0} Y_0 \xrightarrow{g_0} Z_0 \rightarrow 0$
in $\rmod A$ without split summands.

We construct perfect exact sequences $\eta_k$ recursively. 
Let $k \geqslant 0$ such that $\eta_{k}$ has no split summands.
Recall the perfect exact sequence $\tilde{\eta}_k$ from \cref{Lem:ConstructPerfSeqViaARSeq} and the morphism of 
short exact sequences from \cref{Lem:PerfectSeq-Via-PoPb}.
We define $\eta_{k+1}$ to be the sequence obtained from $\tilde{\eta}_{k}$ by removing all split summands. 
Then $\eta_{k+1}$ is a perfect exact sequence without split summands and a direct summand of $\tilde{\eta}_{k}$.
In particular, we have the projection of the short exact sequence onto its direct summand.
In general, the middle morphism is not the natural projection.
\[
\begin{tikzcd}
\eta_k         \ar[d] & 0 \ar[r] & X_k     \ar[d, "s_k"]             \ar[r, "f_k"]                  & Y_k               \ar[d, "\tikzmatez{1}{0}"] \ar[r, "g_k"]                  & Z_k \ar[d, equal]                      \ar[r] & 0 \\
\tilde{\eta}_k \ar[d] & 0 \ar[r] & E_{X_k} \ar[d, twoheadrightarrow] \ar[r, "\tikzmatez{v_k}{t_k}"] & Y_k \oplus T(X_k) \ar[d, twoheadrightarrow]  \ar[r, "\tikzmatze{g_k}{w_k}"] & Z_k \ar[d, twoheadrightarrow, "\pi_k"] \ar[r] & 0 \\
\eta_{k+1}            & 0 \ar[r] & X_{k+1}                           \ar[r, "f_{k+1}"]              & Y_{k+1}                                      \ar[r, "g_{k+1}"]              & Z_{k+1}                                \ar[r] & 0  
\end{tikzcd}
\]
Note that for all $k$, the module $Z_{k+1}$ is a direct summand of $Z_k$ and consequently also of $Z_0$.
Furthermore, if $\eta_k$ is an almost split sequence, $\tilde\eta_k$ will be a split sequence and therefore $\eta_{k+1} = 0$.
In fact, in this case $v_k$ is an isomorphism.
\end{Construction}

In general, this construction will not terminate, as the following example shows.
\begin{Example} \label{Ex:Kronecker}
Let $k$ be an algebraically closed field.
We consider the Kronecker algebra $A$ given by the following quiver. 
\[
\begin{tikzcd}
1 \ar[r, shift left] \ar[r, shift right] & 2
\end{tikzcd} 
\]
Since $A$ is hereditary, every short exact sequence starting in a non-projective module is perfect exact;
cf.\ \cref{Ex:PerfectSeq}.
We denote the indecomposable $A$-modules by their dimension vector.
Let $n \in \ZZ_{\geqslant 0}$. The preprojective component of the Auslander-Reiten quiver of $A$ consists of the modules 
with dimension vector $\tikzmatze{n}{n+1}$. The preinjective component consists of the modules 
with dimension vector $\tikzmatze{n+1}{n}$. Finally, a module with dimension vector $\tikzmatze{n}{n}$ is in the regular
component. In particular, the indecomposable projective modules are given by $P_1 = \tikzmatze{0}{1}$ and
$P_2 = \tikzmatze{1}{2}$.
Recall that the almost split sequence starting in a module in the preprojective component is given by the following. 
\[
0 \to \tikzmatze{n}{n+1} \xrightarrow{s} \tikzmatze{n+1}{n+2} \oplus \tikzmatze{n+1}{n+2} \xrightarrow{t} \tikzmatze{n+2}{n+3} \to 0
\]
We consider the following perfect exact sequence in $\rmod A$.
Note that the starting and middle term of this sequence are in different components of the Auslander-Reiten sequence.
\[
\eta_0 :\quad 0 \to \tikzmatze{2}{3} \xrightarrow{f_0} \tikzmatze{3}{3} \xrightarrow{g_0} \tikzmatze{1}{0} \to 0 
\]
We show by induction, that the perfect exact sequence $\eta_n$ of \cref{Constr:EtaTildeFromEta} is given as follows.
\[
\eta_n :\quad 0 \to \tikzmatze{n+2}{n+3}^{\oplus n+1} \to \tikzmatze{3}{3} \oplus \tikzmatze{n+3}{n+4}^{\oplus n} \to \tikzmatze{1}{0} \to 0 
\]
In fact, we have the following construction step from $n$ to $n+1$ with notation as in \cref{Constr:EtaTildeFromEta}.
\[
\begin{tikzcd}[column sep = .45cm]
\eta_n :       & 0 \ar[r] & \tikzmatze{n+2}{n+3}^{\oplus n+1} \ar[rr, "f_n"] \ar[d, shift right=2.2ex, "s_n"]                               & & \tikzmatze{3}{3} \oplus \tikzmatze{n+3}{n+4}^{\oplus n} \ar[rr] \ar[d, shift right=3ex, hookrightarrow, "\tikzmatez{1}{0}"]                                       & & \tikzmatze{1}{0} \ar[r] \ar[d, equal] & 0 \\
\tilde\eta_n : & 0 \ar[r] & \tikzmatze{n+3}{n+4}^{\oplus 2n+2} \ar[rr, "\tikzmatez{v_n}{t_n}"] \ar[d, shift right=2.2ex, twoheadrightarrow] & & \left(\tikzmatze{3}{3} \oplus \tikzmatze{n+3}{n+4}^{\oplus n}\right) \oplus \tikzmatze{n+4}{n+5}^{\oplus n+1}  \ar[rr] \ar[d, shift right=3ex, twoheadrightarrow] & & \tikzmatze{1}{0} \ar[r] \ar[d, equal] & 0 \\
\eta_{n+1} :   & 0 \ar[r] & \tikzmatze{n+3}{n+4}^{\oplus n+2} \ar[rr, "f_{n+1}"]                                                            & & \tikzmatze{3}{3} \oplus \tikzmatze{n+4}{n+5}^{\oplus n+1} \ar[rr]                                                                                                 & & \tikzmatze{1}{0} \ar[r]               & 0 \\ 
\end{tikzcd}
\]
For all $n \geqslant 0$, we have that $f_{n+1}$ is not irreducible. Thus, the perfect exact sequence $\eta_n$ cannot be 
almost split for any $n \geqslant 0$.

Now, consider the following perfect exact sequence.
\[
\eta_0 : \quad 0 \to \tikzmatze{2}{3} \xrightarrow{f_0} \tikzmatze{3}{4} \xrightarrow{g_0} \tikzmatze{1}{1} \to 0
\]
This time, $f_0$ is an irreducible morphism in the preprojective component. 
However, the ending term is still in a different component of the Auslander-Reiten quiver. 
Similarly to the induction above, we can show that $\eta_n$ of \cref{Constr:EtaTildeFromEta} is given as follows.
\[
\eta_n :\quad 0 \to \tikzmatze{n+2}{n+3} \to \tikzmatze{n+3}{n+4} \to \tikzmatze{1}{1} \to 0 
\]
Again, $\eta_n$ is not an almost split sequence for any $n \geqslant 0$.
\end{Example}

As seen above, we need a condition on both $f_0$ and $g_0$ for \cref{Constr:EtaTildeFromEta} 
to terminate with an almost split sequence.
This condition will be given via the radical of $\rmod A$.
\begin{Definition} \label{Def:Radial-of-Cat}
The \textit{radical} $\rad(\rmod A)$ of $\rmod A$ has the same objects as $\rmod A$ with morphisms 
$f \in \rad_A(X,Y)$ if $gfh$ is not an isomorphism for all $g \in \Hom_A(Z,X)$, $h \in \Hom_A(Y,Z)$ and 
$Z \in \rmod A$ indecomposable.
Recursively, we can define $\rad^0(X,Y) := \Hom_A(X,Y)$ and for $n \in  \ZZ_{\geqslant 1}$
$\rad_A^n(X,Y) := \{ fg : f \in \rad_A(X,Z) \text{ and } g \in \rad^{n-1}_A(Z,Y) \text{ for a } Z \in \rmod A \}$.

Let $f : X \to Y$ in $\rmod A$.
Following \cite{ChaioLiu_Radical}, we say that $f$ has \textit{depth} $n \geqslant 0$, if $f \in \rad_A^n(X,Y)$, 
but $f \not\in \rad^{n+1}(X,Y)$. In case that $f \in \rad_A^n(X,Y)$ for all $n \geqslant 0$, we set $\depth f = \infty$.
\end{Definition}
For more details on the radical see \cite[Section V.7]{AuslanderReitenSmalo}.
We list some properties that are important for our purposes.
\begin{Remark} \label{Rem:Radiacl_Properties}
Let $X,Y \in \rmod A$.
\begin{itemize}
\item[(1)] The $n$-th radical $\rad^n_A(X,Y)$ is a two-sided ideal for $n \in \ZZ_{\geqslant 0}$.   
           Furthermore, we have 
           \[
           \rad^n_A(X,Y) \subseteq \rad^{n-1}_A(X,Y) \subseteq \cdots \subseteq \rad^2_A(X,Y) \subseteq \rad_A(X,Y).
           \]
\item[(2)] Suppose that $f : X \to Y$ has depth zero. Then there exists a $Z \in \rmod A$ indecomposable and morphisms
           $g : Z \to X$, $h : Y \to Z$ such that $gfh$ is an isomorphism. 
           In particular, $g$ and $h$ split so that $Z$ is a common direct summand of $X$ and $Y$.
            
\item[(3)] Let $f : X \to Y$ be a morphism with $\depth f = n$.
           For all morphisms $g : X' \to X$ and $h : Y \to Y'$ in $\rmod A$ we have $\depth(gfh) \geqslant n$ since
           the $n$-th radical is an ideal.

\item[(4)] Suppose that $X$ or $Y$ is indecomposable. 
           An irreducible morphism $f : X \to Y$ has depth $1$.
           If both $X$ and $Y$ are indecomposable, $f : X \to Y$ is irreducible if and only if $\depth f = 1$.
           Furthermore, a direct sum of irreducible morphisms still has depth $1$. 
          
\item[(5)] If $A$ is of finite representation type, every morphism in $\rmod A$ has finite depth; 
           cf.\ \cite[Theorem 7.7]{AuslanderReitenSmalo}.
\end{itemize}
\end{Remark}

We will assume that $f_0\,p$ has finite depth for every projection $p$ onto 
an indecomposable direct summand.
The next result shows that this property gets passed on to all $f_k$ for $k \geqslant 0$.
\begin{Lemma} \label{Lem:f_k_FiniteType}
Suppose given a perfect exact sequence 
$\smash{\eta_0 : 0 \rightarrow X_0 \xrightarrow{f_0} Y_0 \xrightarrow{g_0} Z_0 \rightarrow 0}$ in $\rmod A$ 
without split summands. We use the notation of \cref{Constr:EtaTildeFromEta}.

Suppose that $f_0\,p_0$ has finite depth for every projection $p_0 : Y_0 \to Y_0'$ onto an indecomposable direct summand
$Y_0'$ of $Y_0$.

Then $f_k\,p_k$ has finite depth for all $k \geqslant 0$ and every projection $p_k : Y_k \to Y_k'$ onto an indecomposable 
direct summand $Y_k'$ of $Y_k$.
\end{Lemma}
\begin{proof} 
We proceed by induction on $k$ and show that $\depth(f_{k+1}\,p_{k+1})$ is finite for $k \geqslant 0$.
We use the following expanded notation from \cref{Constr:EtaTildeFromEta}.
\[
\begin{tikzcd}
X_k     \ar[d, "s_k"]                       \ar[r, "f_k"]                  & Y_k               \ar[d, "\tikzmatez{1}{0}"]         \\
E_{X_k} \ar[d, twoheadrightarrow, "\rho_k"] \ar[r, "\tikzmatez{v_k}{t_k}"] & Y_k \oplus T(X_k) \ar[d, twoheadrightarrow, "\tikzmatze{\varphi_k}{\psi_k}"] \\
X_{k+1}                                     \ar[r, "f_{k+1}"]              & Y_{k+1}                                      
\end{tikzcd}
\]
Recall that $0 \to X_k \xrightarrow{s_k} E_{X_k} \xrightarrow{t_k} T(X_k) \to 0$ is a direct sum of almost split sequences
for all $k \geqslant 0$.
Since $\matze{\varphi_k}{\psi_k}$ is split, $Y_{k+1}'$ is either an indecomposable direct summand of $Y_k$ 
or an indecomposable direct summand of $T(X_k)$.

Suppose that $Y_{k+1}'$ is a direct summand of $T(X_k)$.
In this case, we have $\rho_k\, f_{k+1}\, p_{k+1} = t_k \psi_k p_{k+1}$.
Using that $\psi_k p_{k+1}$ is a split epimorphism, $t_k \psi_k p_{k+1}$ is irreducible 
so that $\depth(t_k \psi_k p_{k+1}) = 1$.
This implies $\depth(f_{k+1}\, p_{k+1}) \leqslant \depth(\rho_k\, f_{k+1}\, p_{k+1}) = \depth(t_k \psi_k p_{k+1}) = 1$.

Suppose that $Y_{k+1}'$ is a direct summand of $Y_k$.
In this case, we have $s_k\,\rho_k\, f_{k+1}\, p_{k+1} = f_k\, \varphi_k\, p_{k+1}$.
Using that $\varphi_k p_{k+1}$ is a split epimorphism, 
we know that $\depth(f_k\, \varphi_k\, p_{k+1}) < \infty$ by induction hypothesis.
This implies 
$\depth(f_{k+1}\,p_{k+1}) \leqslant \depth(s_k\rho_k\, f_{k+1}\, p_{k+1}) = \depth(f_k\, \varphi_k\, p_{k+1}) < \infty$.
\end{proof}

We aim to show that \cref{Constr:EtaTildeFromEta} ends with an almost split sequence
under the assumption that $f_0\,p$ and $g_0\,\pi$ have finite depth for every projection $p$ and $\pi$ onto 
an indecomposable direct summand.
We use the assumption on $f_0$ to show that the middle morphism in the construction will eventually 
be an element of the radical. Together with the assumption on $g_0$ this guarantees that we arrive at a split sequence.
Finally, we will use that $\eta_l$ is an almost split sequence if and only if $\tilde\eta_{l}$ is a split sequence.
\begin{Lemma} \label{Lem:ReducePerfSeqToARSeq}
Suppose given a perfect exact sequence 
$\eta_0 : 0 \rightarrow X_0 \xrightarrow{f_0} Y_0 \xrightarrow{g_0} Z_0 \rightarrow 0$ in $\rmod A$ 
without split summands. 
Suppose that $f_0\,p$ and $g_0\,\pi$ have finite depth for every projection $p$ onto an indecomposable direct summand of $Y_0$
and every projection $\pi$ onto an indecomposable direct summand of $Z_0$.

Then there exists an $l \in \ZZ_{\geqslant 0}$ such that $\eta_l$ in \cref{Constr:EtaTildeFromEta} 
is a direct sum of almost split sequences.
\end{Lemma}
\begin{proof}
We first prove that there exists an $l \in \ZZ_{\geqslant 0}$ such that $\tilde{\eta}_{l}$ is a split sequence.
\cref{Constr:EtaTildeFromEta} yields the following sequence of morphisms of short exact sequences.
We aim to show that there is an $l \geqslant 0$ such that $g_0\, \pi_0 \cdots \pi_l = 0$.
\[
\begin{tikzcd}
\eta_0       \ar[d] & 0 \ar[r] & X_0 \ar[d, "s_0"] \ar[r, "f_0"]                                  & Y_0               \ar[d, "\tikzmatez{1}{0}"] \ar[r, "g_0"]                  & Z_0 \ar[d, equal]             \ar[r] & 0 \\
\tilde\eta_0 \ar[d] & 0 \ar[r] & E_{X_0} \ar[d, twoheadrightarrow] \ar[r, "\tikzmatez{v_0}{t_0}"] & Y_0 \oplus T(X_0) \ar[d, twoheadrightarrow]  \ar[r, "\tikzmatze{g_0}{w_0}"] & Z_0 \ar[d, twoheadrightarrow, "\pi_0"] \ar[r] & 0 \\
\eta_1       \ar[d] & 0 \ar[r] & X_1 \ar[d, "s_1"] \ar[r, "f_1"]                                  & Y_1               \ar[d, "\tikzmatez{1}{0}"] \ar[r, "g_1"]                  & Z_1 \ar[d, equal]             \ar[r] & 0 \\
\tilde\eta_1 \ar[d] & 0 \ar[r] & E_{X_1} \ar[d, twoheadrightarrow] \ar[r, "\tikzmatez{v_1}{t_1}"] & Y_1 \oplus T(X_1) \ar[d, twoheadrightarrow]  \ar[r, "\tikzmatze{g_1}{w_1}"] & Z_1 \ar[d, twoheadrightarrow, "\pi_1"] \ar[r] & 0 \\
\eta_2       \ar[d] & 0 \ar[r] & X_2 \ar[d] \ar[r, "f_2"]                                         & Y_2               \ar[d]                     \ar[r, "g_2"]                  & Z_2 \ar[d, equal]             \ar[r] & 0 \\
\vdots              &          & \vdots                                                           & \vdots                                                                      & \vdots                               &
\end{tikzcd}
\]
For $k \geqslant 0$, let $\varphi_k : Y_k \to Y_{k+1}$ be the morphism given by the sequence above.
Assume that $g_0\, \pi_0 \cdots \pi_k = \varphi_0 \cdots \varphi_k\, g_{k+1} \neq 0$ for $k \geqslant 0$. 
In particular, $\varphi_k$ is non-zero for all $k \geqslant 0$.

\textit{Assume} that for all $N \geqslant 0$ there exists a $k \geqslant N$ such that 
$\varphi_N \cdots \varphi_k \in \rad(Y_N,Y_{k+1})$.
Thus, for all $n \geqslant 0$ there exists a $k \geqslant 0$ such that 
$g_0 \pi_0 \cdots \pi_k = \varphi_0 \cdots \varphi_k g_{k+1} \in \rad^n(Y_0, Z_{k+1})$. 
On the other hand, $g_0 \pi_0 \cdots \pi_k$ is non-zero for all $k$ 
and $Z_0$ has only finitely many indecomposable direct summands.
Using that $\pi_k$ is a split epimorphism, this implies that there must exist a 
$k' \geqslant 0$ such that $g_0 \pi_0 \cdots \pi_k \simeq g_0 \pi_0 \cdots \pi_{k'}$ for all $k \geqslant k'$. 
Thus, we have a projection $\pi$ onto an indecomposable direct summand $Z'$ of $Z_0$
such that $g_0\, \pi \in \rad^n(Y_0,Z')$ for all $n$. However, $g_0\,\pi$ has finite depth by assumption.
\textit{A contradiction.}

Therefore, there exists an $N \geqslant 0$ such that for all $k\geq N$ we have $\varphi_N \cdots \varphi_k \not\in \rad(Y_N,Y_{k+1})$.
By \cref{Lem:f_k_FiniteType}, there exists an $1 \leqslant m < \infty$ such that 
$\depth(f_N\, p) < m$ for all projections $p$ of $Y_N$ onto an indecomposable direct summand.

We know that $\varphi := \varphi_N \varphi_{N+1} \cdots \varphi_{N+m-1}$ is neither zero, 
nor in the radical $\rad_A(Y_N,Y_{N+m})$.
Thus, there exists an indecomposable non-zero module $M \in \rmod A$ and morphisms 
$i : M \to Y_N$ and $p : Y_{N+m} \to  M$ such that the composite 
\[
M \xrightarrow{i} Y_N \xrightarrow{\varphi} Y_{N+m} \xrightarrow{p} M
\] 
is an isomorphism. In particular, $\varphi\,p$ is split so that $\depth(f_N\,\varphi\,p) < m$.
By commutativity of the diagram, we have that $f_N\, \varphi \,p$
factors through $s_k$ for $N \leqslant k \leqslant (N+m-1)$.
However, as a direct sum of irreducible morphism, $s_k$ has depth $1$ for $k \geqslant 0$.
We obtain $\depth(f_N\,\varphi\,p) \geqslant m$, a contradiction.

In conclusion, there exists a minimal $l \geqslant 0$ such that 
$\varphi_0 \cdots \varphi_l\, g_{l+1} =  g_0\, \pi_0 \cdots \pi_l$ is zero.
Since $\eta_0$ is not split, the epimorphism $g_0$ is non-zero and we obtain $\pi_0 \cdots \pi_l = 0$.
However, this is a surjection of $Z$ onto its direct summand $Z_{l+1}$.
This implies $Z_{l+1} = 0$ and thus $\eta_{l+1} = 0$ since $\eta_{l+1}$ has no split summands.
By construction, this means that $\tilde{\eta}_l$ is a split sequence.
It remains to show, that $\eta_{l}$ is an almost split sequence.
Suppose that $\tilde{\eta_l}$ is split.
\[
\tilde{\eta_l} : 0 \rightarrow E_{X_l} \xrightarrow{\matez{v_l}{t_l}} Y_l \oplus T(X_l) \xrightarrow{\matze{g_l}{w_l}} Z_l \rightarrow 0
\]
We obtain that $E_{X_l} \oplus Z_l \simeq Y_l \oplus T(X_l)$.
However, $Z_l$ is not a direct summand of $Y_l$ since $\eta_l$ is not a split sequence by construction.
Hence, $Z_l$ is a direct summand of $T(X_l)$. 
Furthermore, the almost split sequence $0 \rightarrow X_l \to E_{X_l} \to T(X_l) \rightarrow 0$ 
has no split direct summands. 
Hence, $E_{X_l}$ and $T(X_l)$ have no common direct summand.

In conclusion, this results in $Z_l \simeq T(X_l)$ and $Y_l \simeq E_{X_l}$ via some isomorphism $\psi : E_{X_l} \to Y_l$
with $\matez{\psi}{0} \simeq \matez{v_l}{t_l}$. 
Since $s_l \matez{v_l}{t_l} = \matez{f_l}{0}$, we obtain an isomorphism of short exact sequences.
\[
\begin{tikzcd}
           & 0 \ar[r] & X_l \ar[r, "s_l"] \ar[d, equal] & E_{X_l} \ar[r, "t_l"] \ar[d, "\sim"'sloped, "\psi"] & T(X) \ar[r] \ar[d, dashed, "\sim"'sloped] & 0 \\
\eta_{l} : & 0 \ar[r] & X_l \ar[r, "f_l"]               & Y_l     \ar[r, "g_l"]                               & Z_l  \ar[r]                               & 0
\end{tikzcd}
\]
Hence, $\eta_l$ is the direct sum of all almost split sequences starting in direct summands of $X_{l}$.
\end{proof}

Inductively, we aim to construct perfect exact sequences in $\rmod B$ corresponding to $\eta_k$ for $0 \leqslant k \leqslant l$.
The next lemma will be used as the induction step.
\begin{Lemma}\label{Lem:PerfSeqUnderStableEquiv_SingleStep}
Let $\stmod A \xrightarrow{\alpha} \stmod B$ be a stable equivalence.

Suppose given a perfect exact sequence $\eta : 0 \rightarrow X \xrightarrow{f} Y \oplus P \xrightarrow{g} Z \rightarrow 0$ in $\rmod A$ 
without split summands, where $P \in \proj A$ and $Y$ has no projective summand.
Suppose that $X$ has no node as a direct summand.

Assume furthermore, that there exists a $\tilde{Q} \in \proj B$ such that
\[ 
0 \rightarrow \alpha(E_X) \xrightarrow{\matez{\tilde{v}}{\tilde{t}}} (\alpha(Y) \oplus \tilde{Q}) \oplus \alpha(T(X)) \xrightarrow{\matze{\tilde{g}}{\tilde{w}}} \alpha(Z) \rightarrow 0
\]
is a perfect exact sequence in $\rmod B$ where
$\underline{\tilde{v}} \simeq \alpha(\underline{v})$ and $\underline{\tilde{g}} \simeq \alpha(\underline{g})$; cf.\ \cref{Lem:ConstructPerfSeqViaARSeq}.

Then there exist $\tilde{P} \in \proj B$ and a perfect exact sequence in $\rmod B$
\[
0 \rightarrow \alpha(X) \xrightarrow{\tilde{f}} \alpha(Y) \oplus \tilde{P} \xrightarrow{\tilde{g}} \alpha(Z) \rightarrow 0
\]
with $\underline{\tilde{f}} \simeq \alpha(\underline{f})$ and $\underline{\tilde{g}} \simeq \alpha(\underline{g})$ in $\stmod B$.
\end{Lemma}
\begin{proof}
Note that $X$ has no projective direct summand, since the given perfect exact sequence has no split direct summands.
Recall that we have the perfect exact sequence $0 \rightarrow X \xrightarrow{s} E_X \xrightarrow{t} T(X) \rightarrow 0$
in this case; cf.\ \cref{Rem:SumOfAss}.
By \cref{Prop:AssUnderStableEquiv}, there exists an $\tilde{R} \in \proj B$ such that 
\[
0 \rightarrow \alpha(X) \xrightarrow{\matez{\tilde{s}}{\tilde{\iota}}} \alpha(E_X) \oplus \tilde{R} \xrightarrow{\matze{\tilde{t}}{\tilde{\pi}}} \alpha(T(X)) \rightarrow 0
\]
is a perfect exact sequence in $\rmod B$ with $\tilde{\underline{s}} = \alpha(\underline{s})$. 
By assumption, we have the following perfect exact sequence in $\rmod B$ with $\underline{\tilde{v}} = \alpha(\underline{v})$ and $\underline{\tilde{g}} = \alpha(\underline{g})$.
\[
0 \rightarrow \alpha(E_X) \xrightarrow{\matez{\tilde{v}}{\tilde{t}}} (\alpha(Y) \oplus \tilde{Q}) \oplus \alpha(T(X)) \xrightarrow{\matze{\tilde{g}}{\tilde{w}}} \alpha(Z) \rightarrow 0
\]
\cref{Lem:PerfectSeq-Via-Snake}.(1) now provides the following perfect exact sequence. 
\[
0 \to \alpha(X) \xrightarrow{\matez{\tilde{s}\tilde{v}}{\tilde{\iota}}} (\alpha(Y) \oplus \tilde{Q}) \oplus \tilde{R} \xrightarrow{\matez{\tilde{g}}{-\tilde{\pi}\tilde{w}}} \alpha(Z) \to 0
\]
We have 
$\tilde{s}\, \tilde{v} \simeq \alpha(\underline{s})\alpha(\underline{v}) 
    = \alpha(\underline{s\, v}) = \alpha(\underline{f})$ 
and $\tilde{g} \simeq \alpha(\underline{g})$
in $\stmod B$; cf.\ \cref{Lem:ConstructPerfSeqViaARSeq}.
\end{proof}

We are now ready to prove the main result of this section.
\begin{Theorem} \label{Thm:PerfectSeqUnderStableEquiv}
Let $\alpha : \stmod A \to \stmod B$ be a stable equivalence.

Suppose given a perfect exact sequence $0 \rightarrow X \xrightarrow{f} Y \oplus P \xrightarrow{g} Z \rightarrow 0$ 
without split summands where $X$ has no node as a direct summand, $P \in \proj A$ and $Y$ has no projective direct summand.

Suppose that $f\,p$ and $g\,\pi$ have finite depth for every projection $p$ onto an indecomposable direct summand of $Y$
and every projection $\pi$ onto an indecomposable direct summand of $Z$.
Then there exists a perfect exact sequence 
\[
0 \rightarrow \alpha(X) \xrightarrow{\tilde{f}} \alpha(Y) \oplus \tilde{P} \xrightarrow{\tilde{g}} \alpha(Z) \rightarrow 0
\]
in $\rmod B$ with $\tilde{P} \in \proj B$ such that $\underline{\tilde{f}} \simeq \alpha(\underline{f})$ 
and $\underline{\tilde{g}} \simeq \alpha(\underline{g})$ in $\stmod B$.
\end{Theorem}
\begin{proof}
We denote the given perfect exact sequence by $\eta_0$ and use the notation of \cref{Constr:EtaTildeFromEta}.
By \cref{Lem:ReducePerfSeqToARSeq} there exists an $l \in \ZZ_{\geqslant 0}$ and 
perfect exact sequences $\eta_k$ for $1 \leqslant k \leqslant l$
such that $\eta_l$ is a direct sum of almost split sequences.
Furthermore, $\eta_{k+1}$ is a direct summand of the sequence $\tilde{\eta}_{k}$.
\[
\tilde{\eta}_k : 0 \rightarrow E_{X_k} \rightarrow Y_k \oplus T(X_k) \rightarrow Z \rightarrow 0
\]
By assumption, $X_0$ has no node as a direct summand.
Let $k \geqslant 1$ and assume that $X_k$ has a node as a direct summand. 
Since $X_k$ is a direct summand of $E_{X_{k-1}}$, the node is also
a direct summand of $E_{X_{k-1}}$. However, by \cref{Lem:NodeAsMiddleTermOfAss} 
the middle term of an almost split sequence has no nodes as direct summand. A contradiction.
Thus $X_k$ has no node as a direct summand for all $0 \leqslant k \leqslant l$.

We verify by induction on $0 \leqslant k \leqslant l$ that the assertion holds for $\eta_k$.

Let $k = l$. Then the given perfect exact sequence is a direct sum of almost split sequences and the claim holds by \cref{Prop:AssUnderStableEquiv}.

Let $0 \leqslant k < l$. Suppose that the assertion holds for $\eta_{k+1}$.
We know that $\tilde{\eta}_k$ is the direct sum of $\eta_{k+1}$ and a split exact sequence. 
Therefore, $\alpha$ preserves the perfect exact sequence $\tilde{\eta}_k$ as well.
As a consequence, we can apply \cref{Lem:PerfSeqUnderStableEquiv_SingleStep} to obtain
that $\alpha$ preserves the perfect exact sequence $\eta_k$ and its morphisms.

In conclusion, the assertion holds for all perfect exact sequences $\eta_k$ with $0 \leqslant k \leqslant l$.
In particular, it holds for $\eta_0$.
\end{proof}

As seen in \cref{Ex:Kronecker}, the assumption on the depth of $f$ and $g$ is needed for our proof of this theorem. 
However, it seems unclear whether this assumption is really necessary for the result to hold.

With regard to \cref{Def:PreservePerfSeq_StableEquiv}, we have the following corollary using \cref{Rem:Radiacl_Properties}.
\begin{Corollary} \label{Cor:StableEquivPreservePerfSeq}
Let $A$ and $B$ be finite dimensional algebras without nodes.
Suppose that $A$ and $B$ have finite representation type.

Then every stable equivalence $\alpha : \stmod A \to \stmod B$ and its quasi-inverse preserve perfect exact sequences.
\end{Corollary}


\section{Stable Gorenstein-projective modules}
\label{Sec:Gorenstein}

We apply our previous results to the category of stable Gorenstein-projcetive modules.
We begin with the definition of Gorenstein-projective modules.
\begin{Definition} \label{Def:GorensteinProj}
An $A$-module $X$ is said to be \textit{Gorenstein-projective} if there exists a totally acyclic 
complex $F^\bt \in \KKptac A$ such that $\HH^0(\tleq F^\bt) \simeq X$.
\end{Definition}
Let $\Gproj A$ be the full subcategory of $\rmod A$ consisting of Gorenstein-projective modules.
Let $\stgproj A$ be the full subcategory of $\stmod A$ consisting of Gorenstein-projective modules.
Note that a projective $A$-module $P$ is Gorenstein-projective via the complex $0 \to P \to P \to 0$.

The following lemma collects some facts about the category of Gorenstein-projective modules which can be found in 
\cite[Section 2.1]{Chen_Gorenstein}. The first property implies that every short exact sequence in $\Gproj A$ 
is perfect exact; cf.\ \cref{Lem:Ext=0_PerfectSeq}.
\begin{Lemma} \indent \label{Lem:Facts_GorensteinProjective}
\begin{itemize}
\item[(1)] $\Ext^1_A(X,A) = 0$ for $X \in \Gproj A$.

\item[(2)] The syzygy functor $\Omega : \stgproj A \to \stgproj A$ is a self-equivalence of categories.

\item[(3)] The category $\stgproj A$ is triangulated with suspension $\Omega^{-1}$ and distinguished triangles 
           isomorphic to those induced by short exact sequences.
\end{itemize}
\end{Lemma}

Note that $\KKptac A$ is contained in $\LL_A$; cf.\ \cref{Def:L-Category}.
Using this, a Gorenstein-projective module $X$ can be characterized via its image $F_X^\bt$ in $\LL_A$; cf.\ \cref{Thm:Equivalence_F}.
\begin{Lemma} \label{Lem:CharacterisationGorensteinViaShift}
The following are equivalent for $X \in \rmod A$.
\begin{itemize}
\item[(1)] $X$ is Gorenstein-projective.

\item[(2)] $F_X^\bt \in \KKptac A$.

\item[(3)] $F_X^\bt[k] \in \LL_A$ for all $k \in \ZZ$.
\end{itemize}
\end{Lemma}
\begin{proof}
Suppose that $X$ is Gorenstein-projective. Assume that $X$ is not projective, otherwise $F_X^\bt \simeq 0$ in $\LL_A$.

By definition, there exists a totally acyclic complex $P^\bt \in \KKptac A$ such that 
$\HH^0(\tau_{\leqslant 0} P^\bt) \simeq X$. 
However, such a complex $P^\bt$ is an element of $\LL_A$. 
Therefore, we have that $F_X^\bt \simeq P^\bt$ by \cref{Thm:Equivalence_F}, 
since $\HH^0(\tau_{\leqslant 0} P^\bt) \simeq X = \HH^0(\tau_{\leqslant 0} F_X^\bt)$.
We obtain that $F_X$ is totally acyclic.

Recall that $\HH^0(\tleq F_X^\bt) = X$. Hence, (2) implies (1).
Furthermore, a totally acyclic complex $F^\bt$ satisfies $\HH^k(F^\bt) = 0$ and $\HH_k(F_\bt^\ast) = 0$ for all $k \in \ZZ$.
Thus, (2) also implies (3).

Now, suppose that $F_X^\bt[k] \in \LL_A$ for all $k \in \ZZ$.
We show that $F_X^\bt$ is a totally acyclic complex.
Using that 
$\LL_A := \left\lbrace F^\bt \in \KKp{A} \;\vert\; \HH^{<0}(F^\bt) = 0, \, \HH_{\geqslant0}(F^\ast_\bt) = 0 \right\rbrace$, 
we obtain the following for all $k \in \ZZ$.
\begin{align*}
\HH^k(F_X^\bt) = \HH^{-1}(F_X^\bt[k+1]) &= 0 \\
\HH_k((F_X^\ast)_\bt) = \HH_0((F_X^\ast)_\bt[k]) = \HH_0\left((F_X^\bt[k])^\ast\right) &= 0
\end{align*}
In conclusion, $F_X^\bt \in \KKptac A$.
\end{proof}

We recover that the category of stable Gorenstein-projective modules is equivalent to $\KKptac A$.
See \cite[Theorem 4.4.1]{Buchweitz} or \cite[Proposition 7.2]{Krause} for different approaches.
\begin{Lemma} \label{Lem:InducedEquiv_GorensteinKtac}
The equivalence $\FF : \stmod A \to \LL_A$ restricts to an equivalence of triangulated categories
$\stgproj A \xrightarrow{\sim} \KKptac A$.
\end{Lemma}
\begin{proof}
\cref{Lem:CharacterisationGorensteinViaShift} shows that $\FF$ restricts to an equivalence 
$\stgproj A \to \KKptac A$. It remains to show that this is a triangulated functor.

Using \cref{Lem:Facts_GorensteinProjective}.(1), we know that every short exact sequence in $\Gproj A$ is a
perfect exact sequence by \cref{Lem:Ext=0_PerfectSeq}. By \cref{Prop:PerfSeq-DistTriang}, the functor $\FF$ maps
perfect exact sequences to distinguished triangles in $\KKp A$ and therefore preserves triangles.
In particular, a perfect exact sequence
\[
0 \to \Omega(X) \to P \to X \to 0
\]
with $X \in \Gproj A$ and $P \in \proj A$ corresponds to the following triangle.
\[
F_{\Omega(X)}^\bt \to 0 \to F_X^\bt \to
\]
Thus, we have a natural isomorphism $F_X^\bt[-1] \simeq F_{\Omega(X)}^\bt$ so that $\FF$ commutes with the shift.
\end{proof}

We note that condition (3) in \cref{Lem:CharacterisationGorensteinViaShift} 
can be expressed via the existence of perfect exact sequences with projective middle term.
Recall that a stable equivalence $\stmod A \to \stmod B$ preserves perfect exact sequences if 
$A$ and $B$ are of finite representation type and have no nodes;
cf.\ \cref{Def:PreservePerfSeq_StableEquiv} and \cref{Cor:StableEquivPreservePerfSeq}.
Furthermore, every stable equivalence induced by an exact functor preserves perfect exact sequences with projective
middle term if the inverse equivalence is also induced by an exact functor; cf.\ \cref{Prop:ExactsStableEquiv_ProjMiddleTerm}.
\begin{Lemma} \label{Lem:ShiftUnderStableEquiv}
Let $\alpha : \stmod A \to \stmod B$ be a stable equivalence such that $\alpha$ and its quasi-inverse 
preserve perfect exact sequences with projective middle term. 
Suppose given $X \in \rmod A$ and $k \in \ZZ$.
We have $F_X^\bt[k] \in \LL_A$ if and only if $F_{\alpha(X)}^\bt[k] \in \LL_B$.
\end{Lemma}
\begin{proof}
If $F_X^\bt[1] \in \LL_A$, there exists a $Y \in \rmod A$ such that $F_Y^\bt \simeq F_X^\bt[1]$ in $\KKp A$.
We have a distinguished triangle $F_X^\bt \rightarrow 0 \rightarrow F_Y^\bt \rightarrow $ in $\KKp A$.
By \cref{Prop:PerfSeq-DistTriang}, the triangle induces a perfect exact sequence 
$0 \rightarrow X \rightarrow P \rightarrow Y \rightarrow 0$
with some $P \in \proj A$.
By assumption, we obtain a perfect exact sequence
$0 \rightarrow \alpha(X) \rightarrow Q \rightarrow \alpha(Y) \rightarrow 0$
with some $Q \in \proj B$.
By \cref{Prop:PerfSeq-DistTriang} the sequence induces a distinguished triangle in $\LL_B$.
\[
F_{\alpha(X)} \rightarrow 0 \rightarrow F_{\alpha(Y)} \rightarrow
\]
We obtain that $F_{\alpha(X)}[1] \simeq F_{\alpha(Y)} \in \LL_B$.
Swapping the roles of $X$ and $Y$ shows that $F_X^\bt[-1] \in \LL_A$ implies $F_{\alpha(X)}[-1] \in \LL_B$ as well.
Inductively, we have that $F_X^\bt[k] \in \LL_A$ implies $F_{\alpha(X)}^\bt[k] \in \LL_B$ for all $k \in \ZZ$.

Let $\beta : \stmod B \to \stmod A$ be the quasi-inverse of $\alpha$.
The same argument as above yields $F_{\beta(\alpha(X))}^\bt[k] \in \LL_A$ if $F_{\alpha(X)}[k] \in \LL_B$.
Since $\beta(\alpha(X)) \simeq X$ in $\stmod A$, \cref{Thm:Equivalence_F} shows that 
$F_X^\bt[k] \simeq F_{\beta(\alpha(X))}^\bt[k] \in \LL_A$ for $k \in \ZZ$.
\end{proof}

In the setting of the previous lemma, a stable equivalence restricts to an equivalence between the stable categories
of Gorenstein-projective modules. If $\alpha$ additionally preserves arbitrary perfect exact sequences, this restriction is
a triangulated equivalence.
\begin{Theorem} \label{Thm:InducedEquivOnGorensteinProj}
Let $\alpha : \stmod A \to \stmod B$ be a stable equivalence such that 
$\alpha$ and its quasi-inverse preserve perfect exact sequences.
Then $\alpha$ restricts to a triangulated equivalence
\[
\underline{\mathrm{Gproj}}\, A \to \underline{\mathrm{Gproj}}\, B\,.
\]
\end{Theorem}
\begin{proof}
Let $X \in \underline{\mathrm{Gproj}}\,A$. 
By \cref{Lem:CharacterisationGorensteinViaShift}, we have that $F_X^\bt[k] \in \LL_A$ for all $k \in \ZZ$.
Using \cref{Lem:ShiftUnderStableEquiv} and \cref{Lem:CharacterisationGorensteinViaShift} again,
we obtain that $\alpha(X) \in \underline{\mathrm{Gproj}}\, B$.
In conclusion, $\alpha$ restricts to an equivalence
$\alpha : \underline{\mathrm{Gproj}}\, A \to \underline{\mathrm{Gproj}}\, B\,$.
It remains to show that this is a triangulated functor.

By \cref{Lem:Facts_GorensteinProjective}.(1) all short exact sequences of Gorenstein-projective modules are
perfect exact, so that $\alpha$ preserves short exact sequences and thus distinguished triangles.
We show that $\alpha(\Omega(X)) \simeq \Omega(\alpha(X))$ in $\stgproj B$ for all $X \in \stgproj A$.

Let $0 \to \Omega(X) \to P \to X \to 0$ be a short exact sequence without split summands where $P$ is
the projective cover of $X$. We know that this sequence lies in $\Gproj A$ and therefore
must be a perfect exact sequence. By assumption, we obtain a perfect exact sequence
\[
0 \to \alpha(\Omega(X)) \to \tilde{P} \to \alpha(X) \to 0
\]
in $\Gproj B$ with $\tilde{P} \in \proj B$. 
By \cref{Prop:PerfSeq-DistTriang}, this induces the following distinguished triangle.
\[
F_{\alpha(\Omega(X))}^\bt \to 0 \to F_{\alpha(X)}^\bt \to
\]
Thus, we have a natural isomorphism $F_{\alpha(\Omega(X))}^\bt \simeq F_{\alpha(X)}^\bt[-1]$.
On the other hand, consider the short exact sequence 
$0 \to \Omega(\alpha(X)) \to Q \to \alpha(X) \to 0$ with $Q$ the projective cover of $\alpha(X)$ in $\rmod B$.
As above, this is a perfect exact sequence and therefore induces the following distinguished triangle.
\[
F_{\Omega(\alpha(X))}^\bt \to 0 \to F_{\alpha(X)}^\bt \to
\]
Thus, we also have a natural isomorphism $F_{\alpha(X)}^\bt[-1] \simeq F_{\Omega(\alpha(X))}^\bt$.
Together, we obtain a natural isomorphism $F_{\alpha(\Omega(X))}^\bt \simeq F_{\Omega(\alpha(X))}^\bt$ in $\LL_A$.
Finally, this induces the claimed natural isomorphism $\alpha(\Omega(X)) \simeq \Omega(\alpha(X))$ in $\stgproj B$.
\end{proof}

\section{Stable equivalences of Morita type}
\label{Sec:StEquivMoritaType}

Suppose given bimodules $M$ and $N$ that induce a stable equivalence between $A$ and $B$.
If $M$ and $N$ have no projective bimodule as direct summand, componentwise application of $-\otimes_A M$ induces an 
equivalence $-\otimes_A M : \LL_A \to \LL_B$; cf.\ \cite[Theorem 6.5]{Nitsche_TriangulatedHull}.

We aim to show that any equivalence $- \otimes_A M : \LL_A \to \LL_B$ 
with an arbitrary bimodule $\!{}_BM_A$ induces a stable equivalence of Morita type $\stmod A \to \stmod B$.
The proof is based on \cite[Theorem 2.9]{DugasMartinezVilla_MoritaType} by Dugas and Mart\'inez-Villa.
However, we need a slightly different version of this theorem, 
with $\!{}_BN_A = \Hom_B(M,B)$ instead of $\!{}_BN_A = \Hom_A(M,A)$.
\begin{Corollary}[Dugas, Mart\'inez-Villa] \label{Cor:ExactFunctorIsStableEquivOfMoritaType}
Let $A$ and $B$ be finite dimensional $k$-algebras whose semisimple quotients are separable.
Suppose that ${}_AM_B$ is projective as left $A$- and as right $B$-module such that 
$-\otimes_A M$ induces a stable equivalence $\stmod A \to \stmod B$.

If $\!{}_BN_A := \Hom_B(M,B)$ is projective over $A$, then $M$ and $N$ induce a stable equivalence of Morita type between 
$A$ and $B$.
\end{Corollary}
\begin{proof}
Using that $M_B$ is projective, we have the following sequence of natural isomorphisms for all $Y \in \rmod B$.
\begin{align*}
\Hom_B(X \otimes_A M, Y_B) 
    &\simeq \Hom_A(X_A, \Hom_B(M_B,Y_B)_A)
    \simeq \Hom_A(X_A, Y \otimes_B \Hom_B(M,B)_A) \\
    &\simeq \Hom_A(X_A, Y \otimes_B N_A)
\end{align*}
Thus, $-\otimes_A M$ is left adjoint to $-\otimes_B N$.
By a result of Auslander and Kleiner in \cite[Proposition 1.1]{AuslanderKleiner}, we obtain that
$- \otimes_A M$ is left adjoint to $- \otimes_B N$ in $\stmod A$ since $- \otimes_A M$ and $- \otimes_B N$ 
take projective modules to projective modules.
Hence, $- \otimes_B N : \stmod B \to \stmod A$ is the quasi-inverse of 
$- \otimes_A M : \stmod A \to \stmod B$. In particular, $- \otimes_A N$ induces a stable equivalence.

We set ${}_B\tilde{M}_A := \!{}_BN_A$ and ${}_A\tilde{N}_B := \Hom_B(N,B)$.
Then ${}_A\tilde{N}_B \simeq {}_AM_B$ as bimodules and $\tilde{M}$ and $\tilde{N}$
are are projective on both sides.
The result follows by applying \cite[Theorem 2.9]{DugasMartinezVilla_MoritaType} to $\tilde{M}$ and $\tilde{N}$ while 
switching the role of $A$ and $B$. 
\end{proof}

\begin{Remark} \label{Rem:Separability}
Suppose given a projective bimodule ${}_A P_A$. Then $X \otimes_A P_A$ is a projective $A$-module
for all $X \in \rmod A$. The converse does not hold in general. 

However, it does hold, if we assume that 
the semisimple quotients of $A$ and $B$ are separable; cf.\ \cite[Corollary 3.1]{AuslanderReiten_Transpose}
and also \cite[Theorem 2.8]{DugasMartinezVilla_MoritaType}.
This separability assumption is satisfied in the following cases among others.
\begin{itemize}
\item[$\bt$] $k$ is a perfect field.
\item[$\bt$] $A$ and $B$ are given by quivers with relations.
\end{itemize}
\end{Remark}

Note that $\!{}_BN_A$ is projective over $A$ if and only if $\!{}_BN \otimes_A -$ is an exact functor.
We aim to use that a complex $F^\bt \in \LL_A$ can be thought of as a projective resolution $F^{\leqslant 0}$ in $\rmod A$
and a projective resolution $F_{\geqslant -1}^{\ast}$ in $A$-mod. 
Additionally, we need the following result.
\begin{Lemma} \label{Lem:NatIsom_BimodulesHomDual}
Suppose that $M$ is an $A$-$B$-bimodule. Let ${}_BN_A := \Hom_B({}_AM_B,B)$.

For every $P \in \proj A$ there exists a natural isomorphism of left $B$-modules
\[
(P \otimes_A M_B)^\ast \simeq {}_B N \otimes_A P^\ast.
\]
\end{Lemma}
\begin{proof}
We have the following natural isomorphism of left $B$-modules.
\begin{align*}
(P \otimes_A M_B)^\ast &= \Hom_B(P \otimes_A M_B, {}_BB_B) \\
                  &\simeq \Hom_A(P_A, \Hom_B(M_B,{}_BB_B))   \\
                       &= \Hom_A(P_A,{}_BN_A)              
\end{align*}
We show hat $\Hom_A(P_A,{}_BN_A) \simeq {}_B N \otimes_A P^\ast$ using that $P_A$ is projective 
and that $N_A$ is finitely generated.

Since ${}_BB \otimes_k A_A$ is projective as a right $A$-module and since $P_A$ is finitely generated, 
we have the following natural isomorphism of left $B$-modules.
\[
\Hom_A(P_A, {}_BB \otimes_k A_A) \simeq {}_BB \otimes_k A_A \otimes_A P^\ast
\]
Moreover, $\Hom_A(P_A, {}_BB \otimes_k A_A^{\oplus n}) \simeq {}_BB \otimes_k A^{\oplus n} \otimes_A P^\ast$ 
for all $n \in \ZZ_{\geqslant 1}$.

Let $N_A = \left\langle g_1,\dots, g_n \right\rangle$ be a minimal generating system of $N$ as a right $A$-module 
with $n \in \ZZ_{\geqslant 1}$. Consider the following surjective map.
\begin{align*}
B \otimes_k A^{\oplus n}   &\xrightarrow{\varphi} {}_B N_A \\
b \otimes (a_1, \dots, a_n) &\mapsto b \sum\limits_{k=1}^n g_k a_k
\end{align*}
Note that $\varphi$ is a morphism of $B$-$A$-bimodules, so that $\Ker(\varphi)$ is also a $B$-$A$-bimodule.
In particular, $\Ker(\varphi)$ is an $A$-submodule of $B \otimes_k A^{\oplus n}$. Since $A$ is finite dimensional,
$\Ker(\varphi)$ is finitely generated as a right $A$-module.
Therefore, there exists an $m \in \ZZ_{\geqslant 1}$ and a surjective morphism 
$B \otimes_k A^{\oplus m} \to \Ker(\varphi)$ as above.

We obtain a presentation $B \otimes_k A^{\oplus m} \to B \otimes_k A^{\oplus n} \to {}_B N_A \to 0$ of $N$ via bimodules.
Consider the following commutative diagram with exact rows. For the upper row we have used that $P_A$ is projective.
\[
\begin{tikzcd}
\Hom_A(P, B \otimes_k A^{\oplus m})       \ar[r] \ar[d, "\sim" sloped] & \Hom_A(P, B \otimes_k A^{\oplus n})        \ar[r] \ar[d, "\sim" sloped] & \Hom_A(P, N)        \ar[r] \ar[d, dashed, "\sim" sloped] & 0 \\
B \otimes_k A^{\oplus m} \otimes_A P^\ast \ar[r]                        & B \otimes_k A^{\oplus n} \otimes_A P^\ast  \ar[r]                        & N \otimes_A P^\ast  \ar[r]                                & 0
\end{tikzcd}
\]
By the above, the two morphisms on the left are isomorphisms.
Therefore, we obtain a natural isomorphism of left $B$-modules 
$\Hom_A(P_A,{}_BN_A) \simeq {}_B N \otimes_A P^\ast$.
\end{proof}

We are now ready to state the main result of this section.
\begin{Theorem} \label{Thm:EquivOnLL-MoritaType} 
Let $A$ and $B$ be finite dimensional $k$-algebras whose semisimple quotients are separable.

Suppose given a bimodule ${}_A M_B$ such that applying $-\otimes_A M$ componentwise induces an equivalence 
$\LL_A \xrightarrow{\sim} \LL_B$.
Let ${}_BN_A := \Hom_B(M,B)$. Then $M$ and $N$ induce a stable equivalence of Morita type between $A$ and $B$.
\end{Theorem}
\begin{proof}
We show that $M$ is projective as left $A$- and  as right $B$-module 
and we show that $N$ is projective as left $B$- and right $A$-module. 

Since $-\otimes_A M$ maps projective $A$-modules to projective $B$-modules, we have that $M \in \proj B$.
Moreover, this means that $N \in B$-proj.

Let $X \in \rmod A$. Suppose given a projective resolution $P^\bt \in \KKp A$ of $X$.
Then $\tleq F_X^\bt \simeq P^\bt$ in $\KKp A$.
Using that $-\otimes_A M$ is a right exact functor with image in $\LL_B$, 
we obtain that $F_X^\bt \otimes_A M \simeq F_{X \otimes M}^\bt$ .
Hence, $P^\bt \otimes_A M \simeq \tleq (F_X^\bt \otimes_A M) \simeq \tleq F_{X\otimes M}^\bt$ 
is a projective resolution of $X \otimes_A M$.
Thus, we have $\Tor_i^A(X,M) \simeq \HH^{-i}(P^\bt \otimes_A M) = 0$ for all $i\geq 1$.
This implies that $M$ is projective as a left $A$-module.

Let $Y$ be a left $A$-module. Suppose given a projective resolution $Q^\bt$ of $Y$ in $\mathrm{K}(A\text{-proj})$.
There exists an $X \in \rmod A$ such that $\Tr X = Y$.
Then $\tau_{\geqslant -1} F_X^{\bt,\ast} \simeq Q^\bt$ in $\mathrm{K}(A\text{-proj})$.
By \cref{Lem:NatIsom_BimodulesHomDual} we have that  
$\tau_{\geqslant -1}(N \otimes_A F_X^{\bt,\ast}) \simeq \tau_{\geqslant -1} (F_X^\bt \otimes_A M)^\ast$ 
as complexes. 
\[
\begin{tikzcd}
\cdots \ar[r] & N \otimes_A F_1^\ast       \ar[r] \ar[d, "\sim" sloped] & N \otimes_A F_0^\ast       \ar[r] \ar[d, "\sim" sloped] & N \otimes_A F_{-1}^\ast       \ar[r] \ar[d, "\sim" sloped] & N \otimes_A \Tr(X) \ar[r] \ar[d, dashed, "\sim" sloped] & 0 \\
\cdots \ar[r] & (F^1 \otimes_A M)^\ast     \ar[r]                       & (F^0 \otimes_A M)^\ast     \ar[r]                       & (F^{-1} \otimes_A M)^\ast     \ar[r]                       & \Tr(X \otimes_A M) \ar[r]                       & 0 
\end{tikzcd}
\]
Since $F_X^\bt \otimes_A M \in \LL_B$, we have that $\tau_{\geqslant -1} (F_X^\bt \otimes_A M)^\ast$ 
is a projective resolution of $\Tr(X\otimes M)$. Hence 
\[
N \otimes_A Q^\bt \simeq \tau_{\geqslant -1}(N \otimes_A F_X^{\bt,\ast}) \simeq \tau_{\geqslant -1} (F_X^\bt \otimes_A M)^\ast
\]
is a projective resolution of $N \otimes_A \Tr(X)$.
Thus, we have $\Tor_i^A(N,\Tr(X)) \simeq \HH^{-i}(N \otimes_A Q^\bt) = 0$ for $i \geqslant 0$.
This implies that $N$ is projective as a right $A$-module.

Since the functor $-\otimes_A M$ induces an equivalence $\LL_A \xrightarrow{\sim} \LL_B$ and is exact, this also induces 
an equivalence $\stmod A \to \stmod B$.
Now, \cref{Cor:ExactFunctorIsStableEquivOfMoritaType} shows that $M$ and $N$ induce a stable equivalence 
of Morita type between $A$ and $B$.
\end{proof}

For the remainder of this section, suppose that $\!{}_BM_A$ is a bimodule which is projective as left $A$- and as right $B$-module such that 
$-\otimes_A M$ induces a stable equivalence $\stmod A \to \stmod B$.

For self-injective algebras, Rickard has shown in \cite[Theorem 3.2]{Rickard_RecentAdvances}
that such a functor is isomorphic to a stable equivalence of Morita type.
As described above, Dugas and Mart\'inez-Villa provide the following generalization for arbitrary algebras which satisfy the 
separability condition; cf. \cite[Theorem 2.9]{DugasMartinezVilla_MoritaType}. 
A stable equivalence that is induced by an exact functor $- \otimes_A M$ is of Morita type 
if and only if $\Hom_A(M,A)$ is projective on both sides.

We aim to give other sufficient conditions for $-\otimes_A M$ to be a stable equivalence of Morita type.
We start with the following corollary to \cref{Thm:EquivOnLL-MoritaType}.
For a stable equivalence induced by an exact functor it remains to check that
$\HH_k\big((F^\bt \otimes_A M)^\ast\big) = 0$ for $F^\bt \in \LL_A$ and $k \geqslant 0$. 
\begin{Corollary} \label{Cor:MoritaType-Iff-Nu}
Let $A$ and $B$ be finite dimensional $k$-algebras whose semisimple quotients are separable. 
Let ${}_AM_B$ be a bimodule which is projective as left $A$- and as right $B$-module such that
$-\otimes_A M$ induces a stable equivalence $\stmod A \to \stmod B$.

Then $M$ and $\Hom_B(M,B)$ induce a stable equivalence of Morita type between $A$ and $B$ if
one of the following equivalent conditions holds.
\begin{itemize}
\item[(1)] There exist natural isomorphisms $\nu_B(P \otimes_A M) \simeq \nu_A(P) \otimes_A M$ of right $B$-modules 
           for every $P \in \proj A$.

\item[(2)] There exists a natural isomorphism $M \otimes_B \Du\!B \simeq \Du\!A \otimes_A M$ of right $B$-modules.
\end{itemize}

\end{Corollary} 
\begin{proof}
Suppose that condition (1) holds. 
Let $F^\bt \in \LL_A$. Note that $-\otimes_A M$ is exact since ${}_AM$ is projective.
By assumption, we have the following for $k \geqslant 0$.
\begin{align*}
                \quad & \HH_k((F^\bt \otimes_A M)^\ast) = 0 \\
\Leftrightarrow \quad & \HH^k(\nu_B(F^\bt \otimes_A M)) = 0 \\
\Leftrightarrow \quad & \HH^k(\nu_A(F^\bt) \otimes_A M) = 0 \\
\Leftrightarrow \quad & \HH^k(\nu_A(F^\bt)) \otimes_A M = 0 \\
\Leftrightarrow \quad & \HH_k(F_\bt^\ast) \otimes_A M = 0
\end{align*}
The last equation holds, since $F^\bt \in \LL_A$. As a result, we have $F^\bt \otimes_A M \in \LL_B$ and
$- \otimes_A M$ induces an equivalence $\LL_A \to \LL_B$.
By \cref{Thm:EquivOnLL-MoritaType}, we obtain that $M$ and $\Hom_B(M,B)$ 
induce a stable equivalence of Morita type between the algebras $A$ and $B$.

It remains to show the equivalence of (1) and (2). We have the following natural isomorphisms of right $B$-modules.
\begin{align*}
&\nu_A(A) \otimes_A M \simeq \Du\Hom_A(A, A) \simeq \Du\!A \otimes_A M \\
&\Du(M \otimes_B \Du\!B) = \Hom_k(M \otimes_B \Du\!B, k) \simeq \Hom_B(M, \Hom_k(\Du\!B,k)) 
    \simeq \Hom_B(M, B) = M_B^\ast
\end{align*}
Using the above, we see that condition (1) implies condition (2) by letting $P = A$.
\[
\Du\!A \otimes_A M \simeq \nu_A(A) \otimes_A M \simeq \nu_B(A \otimes_A M) \simeq \nu_B(M) \simeq M \otimes_B \Du\!B.
\]
Since every projective $A$-module is a direct summand of $A^{\oplus n}$ for some $n \in \ZZ$, this also shows that 
(2) implies (1).
\end{proof}
Note that if $M$ and $\Hom_B(M,B)$ do not have any non-zero projective bimodules as direct summands and if they induce
a stable equivalence of Morita type, then $\nu_B(X \otimes_A M) \simeq \nu_A(X) \otimes_A M$ for every $X \in \rmod A$;
cf.\ \cite[Lemma 3.3]{DugasMartinezVilla_MoritaType} and \cite[Lemma 4.1]{ChenPanXi_MoritaType}.

Now, we aim to use our previous results about perfect exact sequences.
The following theorem by Yoshino provides a way to relate $\HH_k(F_\bt^\ast)$ 
with $(\HH^k(F^\bt))^\ast$ and $\Ext_A^1\!\big(\Cok(d_F^k),A\big)$.
Although the theorem is stated for commutative rings, the proof also holds for arbitrary finite dimensional algebras.
\begin{Theorem}{\rm(\cite[Theorem 2.3]{Yoshino_UnboundedComplexes})} \label{Thm:SequenceWith_H(*)-H()*} 
Suppose given $F^\bt \in \KKp A$ and $M \in \rmod A$.

For all $k \in \ZZ$ there exists an exact sequence
\[
0 \to \Ext_A^1\!\big(\Cok(d_F^k),M\big) \to \HH^k\!\big(\Hom_A(F^\bt,M)\big) \to \Hom_A\!\big(\HH^k(F^\bt),M\big) \to \Ext^2_A\!\big(\Cok(d_F^k),M\big).
\]
\end{Theorem}

Note that $(\HH^k(F^\bt))^\ast = 0$ for $F^\bt \in \LL_A$ and $k \in \ZZ$ if $\ddim A \geqslant 1$,
as follows from \cref{Lem:L_in_Pperp}.
By \cref{Lem:Charact_nu(S)=0}, the vanishing of $S^\ast$ for a simple module $S$ is invariant under
stable equivalences that preserve perfect exact sequences with projective middle term.

\begin{Lemma} \label{Lem:nu(X)=0-iff-nu(S)=0}
Let $Y \in \rmod A$ such that
every short exact sequence $0 \to X \to Y' \to S \to 0$
with $Y'$ a submodule of $Y$ and $S$ a simple $A$-module is a perfect exact sequence.

Then $Y^\ast = 0$ if and only if $S^\ast = 0$ for all composition factors $S$ of $Y$.
\end{Lemma}
\begin{proof}
We proceed by induction on the length of $Y$. 
There is nothing to show for $l(Y) = 1$ so we assume $l(Y) > 1$.
There exists a short exact sequence 
\[
0 \to X \to Y \to S \to 0
\]
with $S$ a simple $A$-module and $l(X) < l(Y)$. By assumption, this sequence is perfect exact.
This implies that
\[
0 \to S^\ast \to Y^\ast \to X^\ast \to 0
\]
is a short exact sequence. Thus $Y^\ast = 0$ if and only if $X^\ast = 0$ and $S^\ast = 0$.
Since $X$ is a submodule of $Y$, we are done by induction.
\end{proof}

Recall that a complex $F^\bt$ in $\LL_A$ satisfies $\HH^k(F^\bt) \in \Pperp_A$ for all $k \in \ZZ$.
If $\ddim A \geqslant 1$, the assumptions of the lemma above hold for the 
cohomology of $F^\bt$ by \cref{Lem:Pperp_PerfSeq_IfDomDim}.
We also have seen that an exact functor $- \otimes_A M$ preserves perfect exact sequences with projective middle term
if and only if $\Ext^1(Z,A) = 0$ implies $\Ext_B^1(Z \otimes_A M, B) = 0$ for all $Z \in \rmod A$;
cf.\ \cref{Prop:ExactsStableEquiv_ProjMiddleTerm}.(1).
This results in the following proposition.
\begin{Proposition} \label{Prop:ExactStableEquiv_PerfSeqDomDim_IsMoritaType}
Let $A$ and $B$ be finite dimensional $k$-algebras whose semisimple quotients are separable.
Assume that $A$ and $B$ have dominant dimension at least $1$.

Suppose given a bimodule ${}_AM_B$ which is projective as left $A$- and as right $B$-module such that 
$-\otimes_A M$ induces a stable equivalence $\stmod A \to \stmod B$. 
Assume furthermore that the following conditions hold.
\begin{itemize}
\item[(1)] The stable equivalence $-\otimes_A M$ and its quasi-inverse 
preserve perfect exact sequences with projective middle term.

\item[(2)] For all simple $A$-modules $S$ whose injective hull is not projective, the image
           $S \otimes_A M$ is a simple $B$-module.
\end{itemize}
Then $M$ and $\Hom_B(M,B)$ induce a stable equivalence of Morita type between $A$ and $B$.
\end{Proposition}
\begin{proof}
Suppose given $F^\bt \in \LL_A$. Using that $-\otimes_A M$ is an exact functor, it remains to show that
$\HH_k((F^\bt \otimes_A M)^\ast) = 0$ for $k \geqslant 0$. In this case, the assertion follows from
\cref{Thm:EquivOnLL-MoritaType}.
By \cref{Thm:SequenceWith_H(*)-H()*}, the vanishing of $\HH_k((F^\bt \otimes_A M)^\ast)$ is implied by 
$\Ext_B^1(\Cok(d_{F\otimes M}^k),B) = 0$ and $\HH^k(F^\bt \otimes_A M)^\ast = 0$. We fix a $k \geqslant 0$.

\textit{We show that $\Ext_B^1(\Cok(d_{F\otimes M}^k),B) = 0$.} 
Since $F^\bt \in \LL_A$, we have $\HH_k(F_\bt^\ast) = 0$. The exact sequence in \cref{Thm:SequenceWith_H(*)-H()*}
now implies that $\Ext_A^1(\Cok(d_F^k),A) = 0$. 
By \cref{Prop:ExactsStableEquiv_ProjMiddleTerm}.(1) and assumption (1), 
we obtain that $\Ext_B^1(\Cok(d_F^k) \otimes_A M,B) = 0$. 
Using that $- \otimes_A M$ is exact, we additionally have that
\[
\Cok(d_F^k) \otimes_A M \simeq \Cok(d_F^k \otimes_A M) = \Cok(d_{F \otimes M}^k).
\]
This results in $\Ext_B^1(\Cok(d_{F\otimes M}^k),B) = 0$.

\textit{We show that $\HH^k(F^\bt \otimes_A M)^\ast = 0$.}
By \cref{Lem:L_in_Pperp}, we have $\HH^k(F^\bt) \in \Pperp_A$.
In particular, we have $\HH^k(F^\bt)^\ast = 0$ since $\ddim A \geqslant 1$.
It suffices to show the following claim.

\textit{Claim.} Let $X \in \rmod A$ with $X \in \Pperp_A$. Then $(X \otimes_A M)^\ast = 0$.

We prove the claim by induction on the length $l := l(X)$ of $X$.
Since $- \otimes_A M$ is exact, we have $l = l(X) = l(X \otimes_A M)$.
Note that we have $X^\ast = 0$ since $\ddim A \geqslant 1$ by assumption.
Moreover, the assumptions of \cref{Lem:nu(X)=0-iff-nu(S)=0}
are satisfied by \cref{Lem:Pperp_PerfSeq_IfDomDim}. 
In particular, we have $S^\ast = 0$ for all composition factors $S$ of $X$.
Note that this implies that $\nu_A(F_S^0)$ is not projective, since $\Hom_A(S,\nu(F_S^0)) \neq 0$.
Hence, $S \otimes_A M$ is a simple $B$-module by assumption (2).

Let $l = 1$ so that $X$ and $X \otimes_A M$ are simple modules.
Thus, $(X \otimes_A M)^\ast = 0$ if and only if $X^\ast = 0$ by \cref{Lem:Charact_nu(S)=0}.(3,4)
since $- \otimes_A M$ preserves perfect exact sequences with projective middle term by assumption (1).
We have seen above, that $S^\ast = 0$ for all composition factors $S$ of $X$.

Let $l > 1$. Suppose given a composition factor $S$ of $X$ together with a short exact sequence
\[
0 \to U \to X \to S \to 0\,.
\]
By \cref{Lem:Pperp_PerfSeq_IfDomDim}, we have $U^\ast = 0$ and this is a perfect exact sequence.
Using that $l(U) < l$, we can assume that $(U \otimes_A M)^\ast = 0$ by induction.
Thus, the induced short exact sequence 
\[
0 \to U \otimes_A M \to X \otimes_A M \to S \otimes_A M \to 0
\]
is perfect exact in $\rmod B$.
In particular, applying $(-)^\ast$, we obtain a short exact sequence in $B$-mod with $(U \otimes_A M)^\ast = 0$.
\[
0 \to (U \otimes_A M)^\ast \to (X \otimes_A M)^\ast \to (S \otimes_A M)^\ast \to 0
\]
As in the case $l=1$, we also have $(S \otimes_A M)^\ast = 0$. This shows that $(X \otimes_A M)^\ast = 0$.
\end{proof}

Let $S$ be a simple module whose injective hull is not projective.
For algebras without nodes, a stable equivalence maps $S$ up to projective direct summands to a simple module.
This follows from a result by  Mart\'inez-Villa in \cite[Proposition 2.4]{MartinezVilla_Properties}.
We slightly adapt his proof to show the following analogue for stable equivalences that are induced by an exact functor.
\begin{Lemma} \label{Lem:ExactStableEquiv_SimpleModules}
Let ${}_AM_B$ be a bimodule that is projective as left $A$- and as right $B$-module such that 
$-\otimes_A M$ induces a stable equivalence $\stmod A \to \stmod B$. 
Suppose that the inverse stable equivalence is also induced by an exact functor.

Let $S$ be a non-projective simple $A$-module with injective hull $I$ such that $I$ is not projective.
We have $S \otimes_A M \simeq S' \oplus P$ such that 
$S'$ is a simple $B$-module and $P \in \proj B$.
\end{Lemma}
\begin{proof}
The stable equivalence $- \otimes_A M$ induces a one-to-one correspondence between the isomorphism classes of 
indecomposable non-projective modules in $\rmod A$ and in $\rmod B$. We denote this correspondence by $\alpha'$.
Let $\pi : I \twoheadrightarrow I/S$ be the natural projection,
which is an irreducible morphism. Since $I$ is not projective, we know that 
$\underline{\pi} \neq 0$ in $\stmod A$. By \cite[Lemma X.1.2]{AuslanderReitenSmalo}, we obtain that the morphism 
$\alpha'(\pi) : \alpha'(I) \to \alpha'(I/S)$ which is induced by $\pi \otimes M$ is irreducible.

Using that the stable equivalence and its quasi-inverse are induced by an exact functor,
$\alpha'(I)$ is an indecomposable injective and non-projective $B$-module; cf.\ \cite[Lemma 3.5]{Liu_ExactFunctor}.
Thus, $S':=\soc(\alpha'(I))$ is a simple $B$-module. We have $S' \subseteq \Ker(\alpha'(\pi))$ since
$\pi\otimes M$ is not a stable isomorphism. This induces a morphism
\[
f : \alpha'(I)/S' \to \alpha'(I/S)
\]
such that $\pi'\, f = \alpha'(\pi)$ with the natural projection 
$\pi' : \alpha'(I) \twoheadrightarrow \alpha'(I)/S'$.
However, $\alpha'(\pi)$ is irreducible and thus $f$ must be a split epimorphism.
Now, consider the natural projection $\pi' : \alpha'(I) \twoheadrightarrow \alpha'(I)/S'$.
Let $\beta'$ be the inverse of the correspondence $\alpha'$. As above, we obtain that
\[
f' : I/S \to \beta'(\alpha'(I)/S')
\]
is a split epimorphism. As a consequence, $\alpha'(I/S)$ is a direct summand of $\alpha'(I)/S'$.
Together with the split epimorphism $f$, this results in $\alpha'(I)/S' \simeq \alpha'(I/S)$.

Write $I \otimes_A M \simeq \alpha'(I) \oplus P$ and $(I/S) \otimes_A M  \simeq \alpha'(I/S) \oplus Q$ 
with $P,\, Q \in \proj B$. Consider the following commutative diagram with $C$ the cokernel of the induced morphism 
$S' \to S \otimes_A M$. 
\[
\begin{tikzcd}
0 \ar[r] & S' \ar[r] \ar[d] & \alpha'(I) \ar[r] \ar[d] & \alpha'(I/S) \ar[r] \ar[d] & 0 \\
0 \ar[r] & S\otimes_A M \ar[r] \ar[d] & I \otimes_A M \ar[r] \ar[d] & (I/S) \otimes_A M \ar[r] \ar[d] & 0 \\
0 \ar[r] & C \ar[r] & P \ar[r] & Q \ar[r]  & 0 
\end{tikzcd}
\]
Since $- \otimes_A M$ is exact and $\alpha'(I/S) \simeq \alpha'(I)/S'$, all rows are short exact sequences.
In particular, the bottom row splits since $Q$ is projective. Thus, $C$ must be projective as well and we obtain
$S\otimes_A M \simeq S' \oplus C$.
\end{proof}

We summarize the results of this section and include situations in which the assumptions are satisfied.
\begin{Theorem} \label{Thm:StableEquivMoritaType-Overview}
Let $A$ and $B$ be finite dimensional $k$-algebras whose semisimple quotients are separable.
Suppose given a bimodule ${}_AM_B$ which is projective as left $A$- and as right $B$-module such that 
$-\otimes_A M$ induces a stable equivalence $\stmod A \to \stmod B$.

Suppose one of the following conditions holds.
           \begin{itemize}
           \item[$(i)$] The functor $- \otimes_A M$ induces an equivalence $\LL_A \to \LL_B$.
           
           \item[$(ii)$] The homology $\HH_k((F^\bt \otimes_A M)^\ast)$ vanishes for $F^\bt \in \LL_A$ and $k \geqslant 0$.
           
           \item[$(iii)$] There exist natural isomorphisms $\nu_B(P \otimes_A M) \simeq \nu_A(P) \otimes_A M$ 
                          for all $P \in \proj A$.
                        
           \item[$(iv)$] There exists a natural isomorphism $M \otimes_B \Du\!B \simeq \Du\!A \otimes_A M$
                         of right $B$-modules.
           
           \item[$(v)$] The algebras $A$ and $B$ have no nodes.
                        At least one of $A$ or $B$ has dominant dimension at least $1$ and finite representation type.
                        Moreover, for all simple $A$-modules $S$ whose injective hull is not projective, the image
                        $S \otimes_A M$ is an indecomposable $B$-module.
                        
           \item[$(vi)$] The algebras $A$ and $B$ have no nodes.
                        At least one of $A$ or $B$ is a Nakayama algebra.
                        Moreover, for all simple $A$-modules $S$ whose injective hull is not projective, the image
                        $S \otimes_A M$ is an indecomposable $B$-module.
                        
           \item[$(vii)$] The algebras $A$ and $B$ have dominant dimension at least $1$.
                        There is a bimodule ${}_B L_A$ which is projective as left $B$- 
                        and right $A$-module and which induces the inverse stable equivalence.
                        Moreover, for all simple $A$-modules $S$ whose injective hull is not projective, the image
                        $S \otimes_A M$ is an indecomposable $B$-module.
           \end{itemize}
Then $M$ and $\Hom_B(M,B)$ induce a stable equivalence of Morita type between $A$ and $B$.
\end{Theorem}
\begin{proof}
If condition (i) holds, we have seen in \cref{Thm:EquivOnLL-MoritaType} that $M$ and $\Hom_B(M,B)$ induce 
a stable equivalence of Morita type between $A$ and $B$. 
Let $F^\bt \in \LL_A$. Since $- \otimes_A M$ is an exact functor, we know that $\HH^k(F^\bt \otimes_A M) = 0$ for
$k \leqslant -1$. Thus, condition (ii) implies condition (i).
By \cref{Cor:MoritaType-Iff-Nu}, condition (iii) and (iv) also imply condition (i).
The last three conditions (v), (vi) and (vii) are a consequence of \cref{Prop:ExactStableEquiv_PerfSeqDomDim_IsMoritaType} 
using the following additional results.

Since $- \otimes_A M : \stmod A \to \stmod B$ is a stable equivalence, $A$ is of finite representation type if and
only if $B$ is of finite representation type. Moreover, by \cite[Theorem 2.3]{MartinezVilla_Properties},
$\alpha$ preserves the dominant dimension if $A$ and $B$ have no nodes. Note that a Nakayama algebra is of 
finite representation type and has dominant dimension at least $1$.
In (iv) and (v) we now use that a stable equivalence between algebras without nodes and of finite representation type 
preserves perfect exact sequences by \cref{Cor:StableEquivPreservePerfSeq}.
In the setting of part (vi), perfect exact sequences with projective middle term are preserved by 
\cref{Prop:ExactsStableEquiv_ProjMiddleTerm}.
Finally, for a simple $A$-module $S$, we have that $S \otimes_A M$ is isomorphic to a direct sum of a simple module and a 
projective module by \cite[Proposition 2.4]{MartinezVilla_Properties} in the setting of part (iv) and (v) and by 
\cref{Lem:ExactStableEquiv_SimpleModules} in the setting of part (vi).
If $S \otimes_A M$ is indecomposable, $S \otimes_A M$ must be isomorphic to a simple $B$-module.
Therefore, both assumptions of \cref{Prop:ExactStableEquiv_PerfSeqDomDim_IsMoritaType} are satisfied if condition
(iv), (v) or (vi) holds.
\end{proof}

\begin{Remark}
Suppose that ${}_A M_B$ is a bimodule such that $- \otimes_A M$ induces an equivalence $\LL_A \to \LL_B$ as in 
part (i) of the previous theorem.
Let $S$ be a simple $A$-module with $S^\ast = 0$.
If $\ddim A \geqslant 1$, this holds for simple $A$-modules whose 
injective hull is not projective. Then $F_S^{\leqslant 0} = F_S^\bt \in \LL_A$ is a projective resolution of 
$S = \HH^0(\tleq F_S^\bt) $; cf. \cref{Lem:Charact_nu(S)=0}. Thus,
$F_S^\bt \otimes_A M \in \LL_B$ is a projective resolution of $S \otimes_A M$.
In particular, $(S\otimes_A M)^\ast = 0$ by \cref{Lem:Charact_nu(S)=0} and we obtain that $S \otimes M$ 
has no projective direct summand. Thus, $S\otimes_A M$ is indecomposable.

Suppose ${}_A M_B$ and ${}_B N_A$ are bimodules that induce a stable equivalence of Morita type.
If ${}_A M_B$ and ${}_B N_A$ are indecomposable as bimodules, we even have that $S \otimes M$ is indecomposable 
for all simple $A$-modules $S$; cf.\ \cite[Lemma 4.4]{KoenigLiu_sms}.

It seems unclear whether the assumption in the previous theorem on the image $S \otimes_A M$ of a simple $A$-module can be 
dropped if we assume that ${}_A M_B$ is an indecomposable bimodule.
\end{Remark}

\textbf{Acknowledgements.} The results of this article are part of the author's PhD thesis \cite{Nitsche_PhD}.
The author would like to thank Steffen Koenig and Yuming Liu for helpful comments and suggestions.

\end{document}